\documentclass[3p]{elsarticle}
\usepackage{amsmath,amsthm,amscd,amssymb,latexsym,upref,stmaryrd}

\usepackage{verbatim}

\newcommand{\bbC}{{\mathbb{C}}}

\newcommand{\bbR}{{\mathbb{R}}}

\newcommand{\cA}{{\mathcal A}}
\newcommand{\cB}{{\mathcal B}}

\newcommand{\cD}{{\mathcal D}}

\newcommand{\cH}{{\mathcal H}}

\newcommand{\cK}{{\mathcal K}}

\newcommand{\cP}{{\mathcal P}}
\newcommand{\cQ}{{\mathcal Q}}

\newcommand{\cT}{{\mathcal T}}

\newcommand{\sF}{{\mathsf F}}
\newcommand{\sH}{{\mathsf H}}
\newcommand{\sY}{{\mathsf Y}}

\renewcommand{\l}{\lambda}

\newcommand{\x}{\xi}

\DeclareMathOperator{\Tr}{Tr}
\DeclareMathOperator{\tr}{tr}

\renewcommand{\Re}{\text{\rm Re}}

\newcommand{\beq}{\begin{equation}}
\newcommand{\enq}{\end{equation}}

\renewcommand{\ge}{\geqslant}
\renewcommand{\le}{\leqslant}

\let\geq\geqslant
\let\leq\leqslant

\newcommand\restr[2]{{
  \left.\kern-\nulldelimiterspace 
  #1 
  \vphantom{\big|} 
  \right|_{#2} 
  }}

\numberwithin{equation}{section}

\renewcommand{\P}{{\mathcal P}}

\allowdisplaybreaks \numberwithin{equation}{section}

\newtheorem{theorem}{Theorem}[section]

\newtheorem{proposition}[theorem]{Proposition}
\newtheorem{lemma}[theorem]{Lemma}
\newtheorem{corollary}[theorem]{Corollary}

\newtheorem{definition}[theorem]{Definition}

\theoremstyle{remark}
\newtheorem{remark}[theorem]{Remark}

\newtheorem{example}[theorem]{Example}

\begin{document}

\numberwithin{equation}{section}
\allowdisplaybreaks

\title{A regularity condition under which integral operators with operator-valued kernels are trace class}

\tnotetext[t1]{This work was funded by the National Science Foundation under DMS-2106203 and DMS-2106157.}

  
 \author[1]{John Zweck\corref{cor1}%
 }
\ead{zweck@utdallas.edu}
  
 \author[2]{Yuri Latushkin%
 }
\ead{latushkiny@missouri.edu}

\author[1]{Erika Gallo%
 }
\ead{Erika.Gallo@utdallas.edu}

\cortext[cor1]{Corresponding author}

\affiliation[1]{organization={Department of Mathematical Sciences, The University of Texas at Dallas},
city={Richardson, TX},
postcode={75080},
country={USA}}

\affiliation[2]{organization={Department of Mathematics, University of Missouri},
city={Columbia, MO},
postcode={65211},
country={USA}}




\date{\today}

\hspace*{-3mm} 
\begin{abstract} 
We study integral operators
on the space of square-integrable functions from a compact set, $X$, to
a separable Hilbert space, $\sH$.  The kernel of such an operator 
takes values in the ideal of Hilbert-Schmidt operators on $\sH$. 
We establish regularity conditions on the kernel under which the 
associated integral operator is trace class. 
First, we  extend Mercer's theorem to operator-valued kernels
by proving that a continuous, nonnegative-definite, Hermitian symmetric
kernel   defines a trace class integral operator on $L^2(X;\sH)$
under an additional assumption.
Second, we show that a general operator-valued kernel 
that is defined on a compact set and that is 
H\"older continuous with H\"older exponent greater than a half 
is trace class provided that 
the  operator-valued kernel  is essentially bounded as a mapping into the space
of trace class operators on $\sH$. 
Finally, when $\dim \sH < \infty$, we show that  an analogous result also holds
for matrix-valued kernels on the real line, provided that an additional 
exponential decay assumption holds.

\end{abstract}

\maketitle

\parskip=0em
\normalsize

\section{Introduction}

In this paper we revisit a 120 year old problem from the birth of functional 
analysis on the trace and determinant of an integral operator. The novelty of our contribution is that the operators we are concerned with are defined in terms
of  operator-valued  kernels rather than the scalar-valued
kernels that are most commonly treated in the classical theory. 
By an operator-valued kernel, $K$,  we 
mean an $L^2$-kernel  that takes values in the ideal, 
$\cB_2(\sH)$, of Hilbert-Schmidt operators on a separable Hilbert space, $\sH$,
that is $K\in L^2(X\times X;\cB_2(\sH))$ where   $X\subset 
\mathbb R^m$ is compact.
In the special case that $\sH$ is one dimensional the kernel is  scalar-valued
kernel and when $\sH$ is finite dimensional it is matrix-valued.
To each operator-valued kernel, $K\in L^2(X\times X;\cB_2(\sH))$, 
there is an associated integral operator,
\begin{equation}
(\cK \psi)(x) \,\,=\,\,
\int_X K(x,y)  \psi(y)\, dy, \qquad\text{for } \psi\in L^2(X;\sH),
\end{equation}
which is a Hilbert-Schmidt operator on $L^2(X;\sH)$, that is 
$\cK\in \cB_2(L^2(X;\sH))$.
The main goal of this paper is to establish a regularity condition on  an operator-valued kernel, $K$,  which ensures
that the integral operator, $\cK$, is trace class.

Trace class operators are compact operators on a Hilbert  
space for which a notion of trace can be defined.\footnote{For a detailed history of the subject and a comprehensive literature review,
we refer the reader to the books of Barry Simon~\cite{Simon} and Gohberg, Goldberg and Krupnik~\cite{GGK}, as well as to 
the recent influential paper of Bornemann~\cite{Bornemann} on the numerical evaluation
of Fredholm determinants.}  
The theory of trace class operators was developed by Schatten and von Neumann
in the 1940's~\cite{schatten1946crossII}. They defined an operator $\mathcal K$ on $\sH$ to be trace class if it is the composition of two Hilbert-Schmidt operators. 
Because $\cB_2(\sH)$ is an ideal, every trace class operator is Hilbert-Schmidt.
Following Simon~\cite{Simon},  we say that
a compact operator, $\cK$, belongs to the $p$-th 
Schatten class, $\cB_p(\sH)$, if 
\begin{equation}
\| \cK \|_{\cB_p(\sH)} \,\,:=\,\, \sum\limits_{\ell=1}^\infty \mu_\ell^p \,\,<\,\,\infty.
\end{equation}
Here, $\{\mu_\ell\}_{\ell=1}^\infty$ is the sequence of  singular values of $\cK$, which are the eigenvalues of the Hermitian-symmetric nonnegative definite operator,  $\cP = (\cK \cK^*)^{1/2}$.
The cases $p=1$ and $p=2$ are the spaces of trace class and Hilbert-Schmidt operators, respectively.
The trace of a trace class operator  is then defined to be 
the sum of its singular values.
Put more succinctly, in Simon's formulation
the concept of a trace class operator inherently
involves the symmetrized operator, $\cP$, associated with $\cK$.

A major reason for the interest in trace class operators is that
$\cB_1(\sH)$ is the space of operators for which the
regular Fredholm determinant is defined. 
Following Simon~\cite{Simon} and Grothendieck~\cite{grothendieck1956theorie},
if $\mathcal K$ is trace class on $\sH$, then
the $n$-th wedge product, ${\Lambda}^n \mathcal K$, is
trace class on $\Lambda^n \sH$ and the infinite series
\begin{equation}
\operatorname{det}_1(\mathcal I + z\mathcal K) \,\, := \,\,
\sum\limits_{n=0}^\infty z^n  \operatorname{Tr} ({\Lambda}^n \mathcal K)
\end{equation}
defines an entire function of $z\in\mathbb C$. As in the finite rank case (see \cite{Simon} for a proof),
\begin{equation}
\operatorname{det}_1(\mathcal I + z\mathcal K) \,\, = \,\,
\prod\limits_{\ell=1}^\infty (1 + z\lambda_\ell),
\end{equation}
where $\{ \lambda_\ell\}_{\ell=1}^\infty$ are the eigenvalues of $\cK$.

If  $\mathcal K$ is Hilbert-Schmidt but not trace class, 
it is still possible to define a Fredholm determinant.
To do so, we first observe that the operator $R_2(\mathcal K) := 
(1 + \mathcal K)e^{-\mathcal K} - 1$ is trace class~\cite{Simon}, since
it is of the form $R_2(\mathcal K) = {\mathcal K}^2 h(\mathcal K)$ for some
entire function, $h$, and since the square of a Hilbert-Schmidt operator
is trace class. The $2$-modified Fredholm determinant is then defined by
\begin{equation}
\operatorname{det}_2(\mathcal I + z\mathcal K) \,\, := \,\,
\operatorname{det}_1(1+ R_2(z\mathcal K) ) 
\,\, = \,\,
\operatorname{det}_1\left( (1+z\mathcal K) e^{-z\mathcal K}\right),
\end{equation}
which is once again an entire function of $z$. 
In this case, the infinite product
\begin{equation}
\operatorname{det}_2(\mathcal I + z\mathcal K) \,\, = \,\,
\prod\limits_{\ell=1}^\infty \left[ (1 + z\lambda_\ell) e^{-z\lambda_\ell}\right]
\label{eq:TCdet2evalue}
\end{equation}
converges.
Furthermore, if $\mathcal K$ is trace class, then both Fredholm determinants are defined,
\begin{equation}
\operatorname{det}_2(\mathcal I + z\mathcal K) 
\,\,=\,\,
\operatorname{det}_1(\mathcal I + z\mathcal K) \, 
e^{-\operatorname{Tr}(\mathcal K)},
\end{equation}
and the zeros of $\operatorname{det}_2(\mathcal I + z\mathcal K)$ and 
$\operatorname{det}_1(\mathcal I + z\mathcal K)$ coincide.

Fredholm determinants
were introduced by Fredholm~\cite{Fred1903} to characterize the 
solvability of integral equations of the second kind, 
$(\mathcal I + z \mathcal K) \boldsymbol \phi = \boldsymbol \psi$.
For applications of Fredholm determinants in mathematical physics, see
Simon~\cite{Simon} and Bornemann~\cite{Bornemann}. 
Our interest in the Fredholm determinants of matrix and operator-valued kernels stems from applications to the
stability of stationary and breather
solutions of nonlinear wave equations such as the 
complex Ginzburg-Landau equation~\cite{PhysicaD116p95,Kap,shen2016spectra,EssSpec}. 
Specifically, the set of eigenvalues of
the linearization of the complex Ginzburg-Landau equation
about  a stationary (soliton) solution
is given by the set of zeros of a Fredholm determinant of a 
Birman-Schwinger operator that is defined in terms of  a matrix-valued
semi-separable Green's  kernel on the real line~\cite{EJF}. 
In a forthcoming paper, we will use Theorem~\ref{TraceClassThmRealLine} below to show that 
this operator is trace class, rather than
simply being Hilbert-Schmidt. Consequently, 
the theory and numerical computation 
of these Fredholm determinants is considerably simplified.
Similarly, we anticipate that operator-valued kernels (with $\dim \sH = \infty$) will
arise for the breather solutions.

In closely related work, Carey et al.~\cite{carey2014jost} showed that certain
 operator-valued separable kernels generate trace class operators
on $L^2((a,b),\sH)$. In a similar vein,  given a  
Schr\"odinger operator with a potential
$V\in L^1(\mathbb R, \cB_1(\sH))$,
they showed that the associated  Birman-Schwinger 
integral operator is a trace class operator on 
$L^2(\mathbb R,\sH)$. They obtained these results by exploiting
the structure of the kernel
to express the integral operator as the composition of Hilbert-Schmidt operators.

Since trace class operators are Hilbert-Schmidt, every trace class operator, $\cK$, on $L^2(X,\sH)$ has a kernel, $K\in L^2(X\times X, \cB_2(\sH))$.
The question of what regularity 
conditions should be imposed on the kernel, $K$, to ensure
that  the operator, $\cK$, is  trace class goes back to  
Fredholm~\cite{Fred1903}, 
and is still of interest today, partly  due to applications 
to reproducing kernel Hilbert spaces which play an important
role in some machine learning algorithms~\cite{aronszajn1950theory,Smale}.

If a scalar kernel is $C^1$ then a simple integration by parts argument
 shows that the corresponding operator is trace class~\cite{Bornemann,lax2002functional}.
 However, the $C^1$ condition can be overly restrictive. For example,
 for the Birman-Schwinger integral operators in which we are interested,
  the kernel, $K=K(x,y)$, is continuous, but not differentiable,
 across the diagonal, $y=x$. 
 Nevertheless, under some reasonable assumptions it can be shown that 
 $K$ is Lipschitz-continuous.

In his seminal 1903 paper~\cite{Fred1903}, 
Fredholm showed that a regular Fredholm determinant can be
defined  for a scalar-valued, H\"older continuous kernel 
with H\"older exponent greater than a half. Although  
he did not yet have the concept of a trace class operator, 
Gohberg, Goldberg and Krupnik~\cite{GGK} showed that
Fredholm's calculations can be used to show that the associated
integral operator is trace class. 
In 1909, Mercer~\cite{mercer1909} showed that a
non-negative definite Hermitian operator on a finite interval
with a continuous scalar-valued kernel  is trace class
and that the eigen-expansion of the kernel converges uniformly.
 It is not possible to drop the definiteness assumption in Mercer's Theorem
 since Carlemann~\cite{carleman1918uber} 
 constructed a $C^0$-function, $k$, whose
 Fourier coefficients are in $\ell_2$ but not in $\ell_1$.
 Consequently, the operator with continuous 
 kernel $K(x,y) = k(x-y)$ is not trace class.  
 Subsequently, for each H\"older exponent, $\gamma \leq 1/2$,
 Bernstein~\cite{bernstein1934convergence} constructed  a function in the H\"older space,
 $C^{0,\gamma}$, with the same property. 
 This example shows that the estimate underlying  in Fredholm's result is sharp. 
 
 In 1965,  Weidmann~\cite{weidmann1966integraloperatoren}
 obtained an alternate proof of Fredholm's result which
  relies on Mercer's Theorem~\cite{mercer1909,smithies1958} 
and on ideas from Fourier analysis due to Hardy and Littlewood~\cite{hardy1928convergence}. 
Brislawn~\cite{brislawn1988kernels,brislawn1991traceable}  extended 
Weidmann's result from kernels on a compact subset of Euclidean
space to a general measure-theoretic setting.
He showed that in the nonnegative definite case 
a kernel is trace class if and only if the
integral over the diagonal of the Hardy-Littlewood maximal function
of the kernel is finite. 
More recently,  Delgado and Ruzhansky~\cite{delgado2021schatten}
obtained sharp conditions on an $L^2$-kernel that 
ensure that the corresponding operator is trace class.
These conditions essentially state that the action of the inverse of the 
Laplacian (or a related unbounded differential operator)
on the kernel lies in a particular Schatten class.

Mercer's theorem underlies the theory of reproducing kernel Hilbert (RKH) spaces ~\cite{aronszajn1950theory,Smale}. Given a Hilbert space, $\sH$, 
an RKH  space is a Hilbert space, $\cH\subset L^2(X;\sH)$, whose
elements are functions,  $f:X\to\sH$,
for which the evaluation operator, $L_x: \cH \to \sH$ given by 
$L_x(f) = f(x)$ is continuous for all $x\in X$. 
In the case that $\sH$ is one-dimensional, 
the Riesz representation theorem implies that 
$\cH$ is characterized by a kernel, $K(x,y)$, with
the property that $L_x(f) = \int_X K(x,y)f(y)\, dy$. 
Recently De Vito, Umanita and Villa~\cite{de2013extension} extended Mercer's theorem to  RKH spaces  of vector-valued
functions. Their focus was on 
establishing the uniform convergence
of the eigen-expansion for matrix-valued kernels that were already 
assumed to be trace class. 
Carmello, De Vito, and Toigo~\cite{carmeli2006vector}
established a version of Mercer's theorem that they used to
study the relationship between bounded,
positive type, operator-valued kernels ($\dim\sH = \infty$)
and the associated RKH space, $\cH$.
Matrix and operator-valued kernels have also been used 
in machine learning applications~\cite{minh2016operator,reisert2007learning}.

In this paper, we 
generalize Mercer's and Weidmann's theorems to obtain
conditions that guarantee that an  operator-valued kernel defines
a trace class operator. 
Our main results  can be summarized as follows.
First, in the special case that $\sH$ is finite dimensional, we will 
prove that  a matrix-valued
kernel that is defined on a compact set  and that is 
H\"older continuous with H\"older 
exponent greater than a half 
gives rise to an integral operator that is trace class.
This result is  a special case of Theorem~\ref{thm:TraceClassInfDimCpt}.
Then, in Theorem~\ref{TraceClassThmRealLine},
we will show that an analogous result also holds
for matrix-valued kernels on the real line, provided that an additional 
exponential decay assumption holds.
While these results are not surprising, we have not been able to find  statements
or proofs of them in the literature. 

However, when $\sH$ is infinite dimensional such a H\"older continuity condition is not enough to guarantee that $\cK$ is trace class. 
In fact, as we show in Example~\ref{ex:MercerCounterExample}, 
it is straightforward to construct a \emph{constant}
kernel that takes values in $\cB_2(\ell_2)$ and which is not trace class. 
(Here $\ell_2$ denotes the space square summable  sequences.)

In the main result of the paper (Theorem~\ref{thm:TraceClassInfDimCpt}), we identify an additional boundedness
condition which ensures that $\cK$ is  trace class.
As we recalled above, the definition of a trace class operator 
involves the symmetrized operator, $\cP$, associated with $\cK$. 
For this reason it not surprising that this boundedness condition
is  a condition on the 
kernel, $P$, of the symmetrized  operator, $\cP$, associated with $\cK$.
Specifically, we assume  that $P \in L^\infty(X\times X; \cB_1(\sH))$.
In other words, there is a constant, $C$, so that for almost all $(x,y)\in X\times X$,
$\| P(x,y) \|_{\cB_1(\sH)} < C$.  We acknowledge that this condition may be hard
to verify since the kernel, $P$, is defined in terms of the eigen-expansion of
$\cK$, which may not be readily available. However, we were not able to
formulate the desired condition directly on the kernel $K$.

In the special case of matrix-valued kernels on a compact set
we obtained two proofs of Theorem~\ref{thm:TraceClassInfDimCpt},
 which are based on the classical proofs 
 of Weidmann~\cite{weidmann1966integraloperatoren} and Fredholm~\cite{Fred1903,GGK}
 for scalar-valued kernels discussed above. 
 While Weidmann's proof can 
 be further extended to operator-valued kernels,  Fredholm's cannot. 
In Section~\ref{sec:Mercer} we state and prove a generalization of
Mercer's Theorem for operator-valued kernels and 
in Section~\ref{sec:OperatorValued} we show how the ideas in Weidmann's paper 
can be extended to prove our main result, Theorem~\ref{thm:TraceClassInfDimCpt}.
In the Appendix, we provide a coarse outline of the generalization
of Fredholm's approach to matrix-valued kernels.

The classical Mercer's Theorem states if $\cP$ is a 
non-negative definite Hermitian operator on a finite interval
with a \emph{continuous} scalar-valued kernel,  then $\mathcal P$ is trace class. 
Furthermore,  there is an orthonormal
set,  $\{\phi_\ell\}_{\ell=1}^\infty$, of continuous eigenfunctions
 and non-negative eigenvalues, $ \{\lambda_\ell\}_{\ell=1}^\infty$ ,
so that the series,
\begin{equation}
P(x,y) \,\,=\,\, \sum\limits_\ell \lambda_\ell \,\phi_\ell(x)\overline{\phi_\ell(y)},
\end{equation}
 converges absolutely and uniformly.
 In Theorems~\ref{thm:MercerA} and \ref{thm:MercerB}, we extend
 Mercer's Theorem to the case of operator-valued kernels. Because of the
 constant kernel counterexample in Example~\ref{ex:MercerCounterExample}, to prove the theorem
  in the infinite-dimensional case we needed to impose an
 additional condition that in essence says that the $\cB_2(\sH)$-valued 
 kernel is actually $\cB_1(\sH)$-valued on the diagonal in $X\times X$. 
Since  Mercer's Theorem is a fundamental result in the theory of reproducing 
 kernel Hilbert spaces, Theorem~\ref{thm:MercerB} 
 may be of independent interest. 
 
In section~\ref{sec:Bochner} we review properties of the  Bochner integral
which extends concept of the Lebesgue integral to Banach-space
valued functions.   In section~\ref{sec:HS} we review material on trace class and Hilbert-Schmidt operators. 
In particular, we establish eigen-expansion formulae for a Hilbert-Schmidt
operator and its kernel on the Bochner space,  $L^2(X;\sH)$. 
In section~\ref{sec:traceformula} we obtain an integral formula for the trace of a trace class
operator on  $L^2(X;\sH)$, which is used in section~\ref{sec:Mercer}.
In section~\ref{sec:Mercer} we state and prove
Mercer's theorem for operator-valued kernels and in section~\ref{sec:OperatorValued}
we establish the main result which gives conditions under which
a  general operator-valued
kernel on a compact set is trace class. In section~\ref{sec:TraceClassProofonReals}, we consider the case of
a matrix-valued kernel on the real line, which is important for applications
to nonlinear wave equations, and in the Appendix, 
we outline an
alternate proof of the main result in the finite dimensional case.

\section{Bochner integrals and Bochner spaces}\label{sec:Bochner}

Let $X\subset \mathbb R^m$ be Lebesgue measurable
and let $\sY$ be a Banach space.
We begin by reviewing the definition of the Bochner spaces, $L^p(X;\sY),$~\cite{Arendt2011,mikusinski1978bochner}.  
A function, $s: X\to\sY$ is \emph{simple} if there are $y_1,\cdots,y_n\in \sY$ and
pairwise disjoint measurable sets, $E_1,\cdots,E_n$ in $X$ of finite measure
so that
\begin{equation}
s(x) \,\,=\,\, \sum_{i=1}^n \chi_{E_i}(x) \, y_i,
\end{equation}
where $\chi_E : X \to \bbR$ is the characteristic function of the set $E$.
In this case,   we define
\begin{equation}
\int_X s\, dx \,\,=\,\, \sum\limits_{i=1}^n \mu(E_i) y_i \,\, \in \,\, \sY,
\end{equation}
where $\mu$ denotes Lebesgue measure.
A function, $f: X\to\sY$, is \emph{strongly measurable} if there are simple functions, $s_n$, so that
\begin{equation}
\lim\limits_{n\to\infty} \| f(x) - s_n(x) \|_\sY = 0, \qquad \text{for almost all }x\in X.
\end{equation}
In this case, the function,  $\| f(\cdot)\|_\sY : X \to \mathbb R^+$, is Lebesgue measurable. 
A strongly measurable function, $f: X\to\sY$, is \emph{Bochner integrable}  if there are simple functions, $s_n$, so that
\begin{equation}
\lim\limits_{n\to\infty} \int_X \| f(x) - s_n(x) \|_\sY \, dx = 0.
\end{equation}
In this case, we define
\begin{equation}
\int_X f\, dx \,\,=\,\, \lim\limits_{n\to\infty}  \int_X s_n\, dx,
\end{equation}
where the limit is taken with respect to the norm on $\sY$. 
We note that $f$ is Bochner integrable if and only if $\int_X \| f(x)\|_\sY \, dx
< \infty$, that is if $\| f(\cdot) \|_\sY \in L^1(X;\mathbb C)$.
In particular, 
\begin{equation}\label{eq:NormIntegralSwitch}
\left\| \int_X f\, dx \right\|_\sY \,\,\leq \,\, \int_X \| f(x) \|_\sY \, dx.
\end{equation} 
Let $p\geq 1$. The \emph{Bochner space}, $L^p(X;\sY)$, is the space of strongly measurable functions, $f:X\to\sY$, so that
\begin{equation}
\| f \|_{L^p(X;Y)} \,\,:=\,\, \left(  \int_X \| f(x) \|^p_\sY \, dx \right)^{1/p} 
\,\, < \,\,\infty.
\end{equation}
We note that $L^p(X;\sY)$ is a Banach space and that  all the standard 
Lebesgue measure theory results in the special case $\sY=\mathbb R$
carry over to general Banach spaces, $\sY$~\cite{mikusinski1978bochner}. 
In particular, by \cite[Theorem 2.4.6-7]{BirmanSolomjak1987}, 
if $\sH$ is a separable Hilbert space, then 
$L^2(X;\sH) $ is a separable Hilbert space with inner product
\begin{equation}
\langle f_1, f_2 \rangle_{L^2(X;H)} \,\,:=\,\,  \int_X \langle f_1(x), f_2(x) \rangle_\sH \, dx.
\end{equation}
Furthermore, if $\{\varphi_\alpha\}_\alpha$ is an orthonormal basis for
$L^2(X,\mathbb R)$ and $\{h_k\}_k$ is an orthonormal basis for $\sH$
then the set of all  $\varphi_{\alpha k} \in L^2(X;\sH)$ defined by
\begin{equation}\label{eq:BasisL2XH}
\varphi_{\alpha k} (x) := \varphi_{\alpha} (x)\, h_k 
\end{equation}
forms an orthonormal basis for $L^2(X,\sH)$.
For future reference we note that if $f\in L^1(X;\sH)$ and $h \in \sH$, then
\begin{equation}\label{eq:SwapIntInnerProd}
\left\langle \int_X f(x)\, dx , g\right\rangle_\sH \,\,=\,\,
\int_X \langle  f(x)\ , g\rangle_\sH, dx.
\end{equation}

 \section{Hilbert-Schmidt operators on  $L^2(X;\sH)$}\label{sec:HS}
 
In this section we summarize some background material
on trace class and Hilbert-Schmidt operators 
on the Bochner space, $L^2(X;\sH)$,
much of which is discussed in Birman and Solomjak~\cite{BirmanSolomjak1987}.

Let $\mathsf{H}$ be a real or complex separable 
Hilbert space with inner product $\langle \cdot, \, \cdot\rangle_\sH$, which is conjugate linear in the first slot and linear in the second slot.  
 We let $\cB(\cdot)$, $\cB_p(\cdot)$, and $\cB_\infty(\cdot)$, 
denote the sets of bounded, $p$-Schatten-von Neumann, and compact
operators, respectively, acting  on a Hilbert space $(\cdot)$.

 Let  $\cK\in \cB_\infty (\sH)$ be  a compact operator on  $\sH$.
Since $\cK\cK^*$ and $\cK^*\cK$ are compact, self-adjoint operators, 
by the spectral theorem there exist orthonormal bases of eigenfunctions 
$\{\psi_\ell\}_{\ell=1}^\infty$ and $\{\phi_\ell\}_{\ell=1}^\infty$ such that 
\begin{align}
\label{KK1}
\cK^*\cK\phi_\ell=\mu_\ell^2\phi_\ell, \,\psi_\ell=\frac1{\mu_\ell}\cK\phi_\ell,\,
\cK\cK^*\psi_\ell=\mu_\ell^2\psi_\ell,\, \phi_\ell=\frac1{\mu_\ell}\cK^*\psi_\ell,
\end{align}
where $\mu_\ell$ are the singular values of the operator $\cK$, that is, the eigenvalues of the operator $\cP=(\cK\cK^*)^{1/2}$, ordered such that $\mu_1\ge\mu_2\ge\dots \geq 0$, $\mu_\ell\to0$ as $\ell\to\infty$.

We define $\cK$ to be \emph{trace class} if
$\| \cK \|_{\cB_1 (\sH)} :=  \sum\limits_{k=1}^\infty \mu_k 
\,\,< \,\, \infty$ and we let $\cB_1 (\sH)$ denote the space of 
trace class operators on $\sH$.
The trace of  a trace class operator is defined  by
\begin{equation}\label{eq:DefTrace}
\operatorname{Tr}(\mathcal K) \,\,=\,\, \sum_k
\langle h_k, \mathcal K h_k \rangle_\sH,
\end{equation}
where here and below $\{h_k\}$ is any orthonormal basis for $\mathcal H$. 
Then~\cite{Simon,lidskii1959non}, 
\begin{equation}\label{eq:TrB1ineq}
| \operatorname{Tr}(\mathcal K) | \,\,\leq\,\, \| \mathcal K \|_{\cB_1(\sH)},
\end{equation}
and
\begin{equation}\label{eq:TraceByEval}
\operatorname{Tr}(\mathcal K) 
 \,\,=\,\, \sum\limits_\ell \mu_\ell \langle \phi_\ell, \psi_\ell\rangle_\sH
 \,\,=\,\,  \sum\limits_\ell \lambda_\ell,
\end{equation}
where $\{ \lambda_\ell \}$ is the set of eigenvalues of 
$\mathcal K$, counted with algebraic multiplicity.

Similarly,
we define $\cK$ to be \emph{Hilbert-Schmidt} if
\begin{equation}\label{eq:HSnormViaSingValues}
\| \cK \|_{\cB_2 (\sH)} := \left( \sum\limits_{k=1}^\infty \mu_k^2 \right)^{1/2} 
\,\,< \,\, \infty,
\end{equation}
 and we let $\cB_2 (\sH)$ denote the space of 
Hilbert-Schmidt operators on $\sH$. We recall that every trace class operator
is Hilbert-Schmidt. 
By \cite[Theorem 11.3.1]{BirmanSolomjak1987}, 
\begin{equation}\label{eq:HSnormViaONB}
\| \cK \|_{\cB_2 (\sH)} = \left( \sum\limits_{k=1}^\infty \| \cK h_k\|^2_\sH
\right)^{1/2}.
\end{equation}
Consequently, $\cB_2 (\sH)$ is  a Hilbert space with inner product
\begin{equation}\label{eq:HSInnerProductViaONB}
\langle \cK_1, \cK_2 \rangle_{\cB_2 (\sH)} := 
\sum\limits_{k=1}^\infty \langle \cK_1h_k, \cK_2 h_k \rangle_\sH.
\end{equation} 
For future use we note that, by \eqref{eq:NormIntegralSwitch},
if $\cK\in L^1(X;\cB_p(\sH))$ (for $p=1, 2$) then
\begin{equation}\label{eq:TriInt}
\left\| \int_X \cK(x)\, dx \right\|_{\cB_p (\sH)} \,\,\leq\,\, 
\int_X \|\cK(x) \|_{\cB_p (\sH)}\, dx.
\end{equation}
Furthermore, we recall that
$\cB_p (\sH)$  is a two-sided ideal in 
$\cB (\sH)$~\cite{Simon}, and that 
\begin{equation}\label{eq:SubMult}
\| \cK_1 \cK_2 \|_{\cB_p (\sH)} \,\,\leq\,\,
\| \cK_1  \|_{\cB (\sH)} \| \cK_2 \|_{\cB_p (\sH)} \,\,\leq\,\,
\| \cK_1  \|_{\cB_p (\sH)} \| \cK_2 \|_{\cB_p (\sH)}.
\end{equation}

Given $a,b\in\sH$ we let $a\otimes b :\sH\to\sH$ be the rank one operator
defined by 
\begin{equation}\label{eq:RankOneOp}
(a\otimes b)h=\langle b, h\rangle_\sH\, a, \qquad h\in\sH.
\end{equation}
By the Riesz representation theorem, every rank one, bounded linear operator 
on $\sH$ is of the form $a\otimes b$ for some $a,b\in\sH$. 
It is easy to check that  if $\{ h_k\}_k$ is  an
orthonormal basis of $\sH$, then $\{ h_k \otimes h_\ell\}_{k,\ell}$
is an orthonormal set in  $\cB_2 (\sH)$. 
Moreover, by  \cite[Theorem 11.3.4]{BirmanSolomjak1987}, this set is a basis
for $\cB_2 (\sH)$.

For future use we summarize some facts about rank one operators.
Since the eigenvalues of  $a\otimes b$ are $\{\langle b,a\rangle_\sH,0\}$, we have that 
\begin{equation}\label{trab}
\tr(a\otimes b)=\langle b,a\rangle_\sH.
\end{equation}
 Furthermore, if $a,b,c,d\in\sH$ then
\begin{equation}\label{prodform} 
(a\otimes b)^*=b\otimes a \text{ and }
\big((a\otimes b)(c\otimes d)\big)h=\langle b,c\rangle_\sH(a\otimes d)h,\, h\in\sH.
\end{equation}
By \eqref{prodform} we have $(a\otimes b)(a\otimes b)^*=\|b\|^2_\sH(a\otimes a)$. If $\|a\|_\sH=1$ then $P_a=a\otimes a$ is the orthogonal projection in $\sH$ onto ${\rm span } \{a\}$.  So, for any $a$, the spectrum of $a\otimes a$ in $\sH$  is $\{\|a\|_\sH^2,0\}$. Therefore, the spectrum of the operator $\big((a\otimes b)(a\otimes b)^*\big)^{1/2}$ is the set $\{\|a\|_\sH\,\|b\|_\sH,0\}$, and so
\begin{equation}\label{normab}
\|a\otimes b\|_{\cB(\sH)}=\|a\otimes b\|_{\cB_1(\sH)}=\|a\otimes b\|_{\cB_2(\sH)}=\|a\|_\sH\|b\|_\sH.
\end{equation}

Next we state some results about  integral operators that are defined by  
operator-valued  kernels~\cite{BirmanSolomjak1987}.
Let $X\subset \mathbb R^m$ be a compact set.
We  consider an operator-valued Hilbert-Schmidt kernel,
\begin{equation}\label{KL2}
K\in L^2(X\times X; \cB_2(\sH)),
\end{equation} whose values, $K(x,y)\in\cB_2(\sH) $,  for $x,y\in X$, 
are Hilbert-Schmidt operators on  $\sH$.
We associate to $K$  the  integral operator, $\cK : L^2(X;\sH) \to L^2(X;\sH)$,
defined by
\begin{equation}\label{defcK}
(\cK \psi)(x)=\int_X K(x,y)\psi(y)\,dy, \qquad \text{for } \psi\in L^2(X;\sH),
\end{equation}
where $K(x,y)\psi(y)$ denotes the action of the  operator, $K(x,y)$,
on the element,  $\psi(y)$, of $\sH$. 
The following results justify these statements. 

\begin{proposition}
Let $K\in L^2(X\times X; \cB(\sH))$ and $\psi\in L^2(X;\sH)$.
Then for almost all $x\in X$, we have $K(x,\cdot)\psi(\cdot) \in L^1(X;\sH)$
and  so
\begin{equation}\label{cKdef2}
(\cK \psi)(x)=\int_X K(x,y)\psi(y)\,dy
\end{equation}
is defined. Moreover, $\cK \in \cB(L^2(X;\sH))$ and 
$\| \cK \|_{\cB(L^2(X;\sH))} \leq \| K \|_{L^2(X\times X; \cB(\sH))}$.
\end{proposition}

The next result~\cite[Theorem 11.3.6]{BirmanSolomjak1987} says that
in order for $\cK$ to be a Hilbert Schmidt operator 
on $L^2(X;\sH)$,  the kernel must take values in $\cB_2(\sH)$ rather than
in $\cB(\sH)$.

\begin{theorem}\label{thm:HSopB2ker}
Let $K\in L^2(X\times X; \cB_2(\sH))$ be an operator-valued Hilbert-Schmidt kernel, and define $\cK$ by \eqref{cKdef2}. 
Then $\cK \in \cB_2(L^2(X;\sH))$ is a Hilbert-Schmidt operator 
on $L^2(X;\sH)$ and 
\begin{equation}\label{eq:HSopB2kerIsom}
\| \cK \|_{\cB_2(L^2(X;\sH))} = \| K \|_{L^2(X\times X; \cB_2(\sH))}.
\end{equation}
Conversely, for every $\cK \in \cB_2(L^2(X;\sH))$ there is a unique 
$K\in L^2(X\times X; \cB_2(\sH))$ for which \eqref{cKdef2} holds.
\end{theorem}

\begin{example}
In the finite dimensional case that $\sH = \bbC^n$, the space of Hilbert-Schmidt
operators is simply the space of square matrices, $\cB_2(\bbC^n) =
\bbC^{n\times n}$, so that $K\in L^2(X\times X;\mathbb C^{n\times n})$ is a matrix-valued Hilbert-Schmidt kernel.
In this case, if we regard $\psi\in L^2(X;\bbC^n)$ as being a column vector, then
$K(x,y)\psi(y)$ is given by matrix-vector multiplication. 
\end{example}

\begin{example}
Let $\sH=\ell_2$ be the space of square-integrable sequences  and let 
$\{\mathbf e_i\}_{i=1}^\infty$ be the standard basis for $\ell_2$. If 
$\psi(x) = \sum\limits_{k=1}^\infty \psi_k(x)\mathbf e_k$ and 
$K(x,y) = \sum\limits_{i,j=1}^\infty K_{ij}(x,y)\mathbf e_i\otimes \mathbf e_j$,
then $K\in L^2(X\times X;\cB_2(\ell))$ if
$\int\limits_{X\times X}\sum\limits_{i,j=1}^\infty |K_{ij}(x,y)|^2 \, dxdy< \infty$ and
\begin{equation}
(\cK \psi)(x) = \int_X \sum\limits_{i,j=1}^\infty K_{ij}(x,y) \psi_j(y) \, dy \, \mathbf e_i
\end{equation}
\end{example}

\begin{example}\label{ex:L2Iso}
Let $X$ and $Y$ be compact sets and let $\sH = L^2(Y;\mathbb C)$.
Then there is an isomorphism, 
$T:  L^2(X\times Y;\mathbb C)\to L^2(X;L^2(Y;\mathbb C))$ given by 
 $T(\Psi)(x) = \Psi(x,\cdot)$ and
$T^{-1}(\psi)(x,y) = \psi(x)(y)$.
The induced isomorphism 
\begin{equation}\label{eq:L2Iso1}
T: \cB_2(L^2(X\times Y;\mathbb C))
\to \cB_2(L^2(X;L^2(Y;\mathbb C)))
\end{equation}
 denoted by $ \cK = T(\widetilde\cK)$
is given by
\begin{equation}
(\cK\psi)(x) = (\widetilde\cK\Psi)(x,\cdot)
= \int_X  K(x,x')\psi(x')\, dx'
\end{equation}
where $\psi \in L^2(X;L^2(Y;\mathbb C))$ and 
\begin{equation}\label{eq:L2Iso2}
 K(x,x')\psi(x') = \int_Y \widetilde K( (x,\cdot),(x',y'))\Psi(x',y')\, dy',
\end{equation}
where $\widetilde K\in L^2((X\times Y) \times (X\times Y);\mathbb C)$ is the kernel
for $\widetilde\cK$ and 
$K\in L^2(X\times X; \cB_2(L^2(Y;\mathbb C)))$
is the kernel for $ \cK$.
\end{example}

\begin{remark}\label{rem:compositionHS}
Let  $K_1$ and $K_2$ denote the kernels of Hilbert-Schmidt operators,
$\cK_1$ and $\cK_2$, on $L^2(X;\cH)$. Then the kernel,
\begin{equation}\label{eq:compositionHSkernels}
L(x,y) \,\,=\,\, \int_X K_1(x,z)\, K_2(z,y)\, dz,
\end{equation}
is in $L^2(X\times X; \cB_2(\sH))$ since by \eqref{eq:TriInt}, \eqref{eq:SubMult},
and the standard H\"older inequality,
$\| L \|_{L^2(X\times X; \cB_2(\sH))}
\leq \| K_1 \|_{L^2(X\times X; \cB_2(\sH))}\,  
\| K_2 \|_{L^2(X\times X; \cB_2(\sH))}$.
Therefore, by Theorem~\ref{thm:HSopB2ker}, $L$ is the kernel of the composite
$\cK_1\cK_2$ operator on $L^2(X;\cH)$.
\end{remark}

\begin{theorem}\label{thm:KernelExpansion}
The kernel, $K$, of the operator $\cK$ in Theorem~\ref{thm:HSopB2ker} 
is given in terms of the singular values, $\mu_\ell$, and eigenfunctions,
$\psi_\ell$ and $\phi_\ell$, in \eqref{KK1} by
\begin{equation}\label{Kform}
K(x,y)\,\,=\,\,\sum_{\ell=1}^\infty\mu_\ell\,\psi_\ell(x)\otimes\phi_\ell(y),
\end{equation}
where the series on the RHS converges  in the norm 
on  $L^2(X\times X;\cB_2(\sH))$, that is,
\begin{equation}\label{KformII}
\int_X\int_X\|K(x,y)-\sum_{\ell=1}^n\mu_\ell\,\psi_\ell(x)\otimes\phi_\ell(y)\|_{\cB_2(\sH)}^2\,dxdy\to0\text{ as $n\to\infty$}.\end{equation}
\end{theorem}

\begin{remark}\label{rem:AdjointHS}
Applying Theorem~\ref{thm:KernelExpansion} to the adjoint operator and invoking \eqref{prodform} we find that the kernel of $\cK^*$ is given by 
\begin{equation}
K^*(x,y) \,\,=\,\, \sum_{\ell=1}^\infty\mu_\ell\,\phi_\ell(x)\otimes\psi_\ell(y)
 \,\,=\,\, \left[\sum_{\ell=1}^\infty\mu_\ell\,\psi_\ell(y)\otimes\phi_\ell(x)\right]^*
\,\,=\,\, [K(y,x)]^*.
\end{equation}
\end{remark}

The following corollary  follows from Theorem~\ref{thm:HSopB2ker}, Remarks~\ref{rem:AdjointHS} and \ref{rem:compositionHS}, and 
\eqref{KK1} and \eqref{prodform}.

\begin{corollary}\label{cor:KernelExpansion}
Let $\cK \in \cB_2(L^2(X;\sH))$ be a Hilbert-Schmidt operator 
on $L^2(X;\sH)$ with  singular values, $\mu_\ell$, and eigenfunctions,
$\psi_\ell$ and $\phi_\ell$, as in \eqref{KK1}.
Then the series expansions 
\begin{align}
\cK \,\,=\,\, \sum\limits_\ell \mu_\ell \, \psi_\ell \otimes \phi_\ell,
\qquad & \qquad
\cK^* \,\,=\,\, \sum\limits_\ell \mu_\ell \, \phi_\ell \otimes \psi_\ell, \\
\cK^*\cK \,\,=\,\, \sum\limits_\ell \mu_\ell^2 \, \phi_\ell \otimes \phi_\ell, 
\qquad & \qquad
\cK\cK^* \,\,=\,\, \sum\limits_\ell \mu_\ell^2 \, \psi_\ell \otimes \psi_\ell,
\end{align}
converge in the norm on $\cB_2(L^2(X;\sH))$
and the corresponding expansions of the associated kernels  converge in 
 $L^2(X\times X;\cB_2(\sH))$.
\end{corollary}

\begin{proof}[Proof of Theorem \ref{thm:KernelExpansion}]
Applying the spectral theorem to the self-adjoint operator, $\cK^*\cK$, 
we find that every $\psi\in L^2(X;\sH)$ can be expressed as 
 $\psi = \sum_\ell (\phi_\ell\otimes\phi_\ell)(\psi) +\phi'$,
with $\phi' \in \operatorname{Ker}(\cK^*\cK) = 
\operatorname{Ker}(\cK)$, where the series 
converges in the $L^2(X;\sH)$-norm.
Therefore, by \eqref{KK1},  the series 
$\cK\psi=\sum_\ell\mu_\ell (\psi_\ell\otimes\phi_\ell)(\psi)$
also converges in  $L^2(X;\sH)$.
Let $\cK_{n}=\sum_{\ell=1}^{n}\mu_\ell\,\psi_\ell\otimes\phi_\ell$.
Then for $n>m$, 
\begin{equation}
\| (\cK_{n}-\cK_m)\psi \|^2_{L^2(X;\sH)} \,\,=\,\, \sum_{\ell=m+1}^{n} \mu_\ell^2 \,
| \langle \phi_\ell, \psi\rangle_{L^2(X;\sH)} |^2 
\,\,\leq\,\, \mu_{m+1}^2 \,\|\psi\|^2_{L^2(X;\sH)},
\end{equation}
which implies that $\cK_n$ is a Cauchy sequence in the $\cB(L^2(X;\sH))$-norm.
Therefore, 
\begin{equation}\label{sK}
\cK \,\, =\,\, \sum\limits_\ell \mu_\ell \psi_\ell \otimes \phi_\ell 
\end{equation}
converges in the operator norm of $\cB(L^2(X;\sH))$.

Furthermore,  the series in \eqref{sK} converges in  $\cB_2(L^2(X;\sH))$, since
we claim that
 \begin{equation}\label{B2KMN}
\|\cK_{n} - \cK_m\|_{\cB_2(L^2(X;\sH))}^2\,\,=\,\,\sum_{\ell=m+1}^{n}\mu_\ell^2,
\end{equation}
which can be made arbitarily  small provided $m,n$ are large enough.
  To prove~\eqref{B2KMN}, we use the analog of \eqref{prodform} for $L^2(X;\sH)$ and the orthonormality of  $\{\psi_\ell\}_\ell$ in $L^2(X;\sH)$, 
 to show that $\cK_{n,m} := \cK_n - \cK_m$ satisfies 
\begin{align}\label{KNM}
\cK_{n,m}^*\cK_{n,m}&=\sum_{\ell=m+1}^{n}\sum_{\ell'=m+1}^{n}\mu_{\ell'}\mu_\ell\,(\phi_{\ell'}\otimes\psi_{\ell'})(\psi_\ell\otimes\phi_\ell)\nonumber \\&
=\sum_{\ell=m+1}^{n}\sum_{\ell'=m+1}^{n}\mu_{\ell'}\mu_\ell\,\langle\psi_{\ell'},\,\psi_{\ell}\rangle_{L^2(X;\sH)}(\phi_{\ell'}\otimes\phi_\ell)\nonumber 
\\&=\sum_{\ell=m+1}^{n}\mu_\ell^2\,\phi_\ell\otimes\phi_\ell.
\end{align}
Since \{$\phi_\ell\}_\ell$ is an orthonormal set  in $L^2(X;\sH)$, the operators $\phi_\ell\otimes\phi_\ell$ are mutually orthogonal projections in $L^2(X;\sH)$ onto ${\rm span } \{ \phi_\ell\}$. Therefore the orthogonal decomposition \eqref{KNM} of the operator $\cK_{n,m}^*\cK_{n,m}$ shows that the spectrum of the operator $\cK_{n,m}^*\cK_{n,m}$ is the set $\{\mu_\ell^2: \ell=m+1,\dots,n\}$. So, the set of singular values of $\cK_{n,m}$ is $\{\mu_\ell: \ell=m+1,\dots,n\}$,  which proves the claim \eqref{B2KMN} and finishes the proof of the fact that the series on the RHS of \eqref{sK} converges in $\cB_2(L^2(X;\sH))$.

Finally, let $K_n=K_n(x,y)$ be the kernel for the operator, $\cK_n$. 
By Theorem~\ref{thm:HSopB2ker}
and \eqref{B2KMN}, $K_n$ is Cauchy in $L^2(X\times X;\cB_2(\cH))$. 
Therefore the
series on the RHS of \eqref{Kform} converges in $L^2(X\times X;\cB_2(\cH))$ and \eqref{KformII} holds. 
\end{proof}

\section{A formula for the trace of an operator-valued kernel}
\label{sec:traceformula}

In this section we establish an integral formula for the trace of a trace class
operator on $L^2(X;\sH)$. 

\begin{theorem}\label{thm:IntegralFormulaForTrace}
Let $\mathcal K\in \cB_1(L^2(X;\sH))$ be a trace class operator
with a  continuous kernel that  takes values in $\cB_1(\sH)$, i.e, $K\in C^0(X\times X;\cB_1(\sH))$. Then
\begin{equation}\label{eq:IntegralFormulaForTrace}
\operatorname{Tr}(\cK)\,\,=\,\, \int_X \operatorname{Tr}(K(x,x))\, dx.
\end{equation}
\end{theorem}

\begin{proof}
The proof is based on the proof in  the scalar case in Simon~\cite[Theorem 3.9]{Simon}.
For simplicity, we consider the case that $X=[0,1]$. However, the proof can 
easily be extended first to $X=[0,1]^m$ and then
to any compact set $X\subset \mathbb R^m$. 

Let $\{\phi_{n,\ell}\}_{\ell=1}^{2^n}$ be the orthonormal set in $L^2([0,1],\mathbb{C})$ given by
\begin{equation}
    \phi_{n,\ell}(x) = 
    \begin{cases}
     2^{n/2}, \qquad & \frac{\ell-1}{2^n} < x < \frac{\ell}{2^n}.\\
     0, \qquad & \text{otherwise.}
    \end{cases}
\end{equation}
Let $\phi_{n,\ell;k}(x)  = \phi_{n,\ell}(x)\, h_k$, where $\{h_k\}_k$ is an 
orthonormal basis for $\sH$. Let $\cB_n = \operatorname{Span}\{ 
\phi_{n,\ell;k} \, : \, \ell=1,\dots,2^n, \, k=1,2,\dots \}$ and let
$\mathcal{P}_n$ be the orthogonal projection from $L^2(X;\sH)$ 
onto $\cB_n$. Then we can construct  
an orthonormal basis $\{\psi_{\ell,k}\}_{\ell,k=1}^{\infty}$ for $L^2(X;\sH)$ so that 
$\{ \psi_{\ell,k} \, : \, \ell = 1,\dots, 2^n, \, k=1,2,\dots \}$ 
is an orthonormal basis for  $\cB_n$.  
Then,
\begin{equation}
    \Tr(\mathcal{K}) = \lim_{n \rightarrow \infty} \Tr(\mathcal{P}_n \mathcal{K} \mathcal{P}_n).
\end{equation}
Since the trace can be computed in terms of any orthonormal basis, 
\begin{align}
\operatorname{Tr}(\mathcal{P}_n \mathcal{K} \mathcal{P}_n) 
\,\,&=\,\, 
\sum_{\ell=1}^{2^n} \sum_{k=1}^\infty
 \langle \phi_{n,\ell;k}, \mathcal{K} \phi_{n,\ell;k} \rangle_{L^2(X;\sH)}
 \nonumber
 \\
 \,\,&=\,\,
 \sum_{\ell=1}^{2^n} \sum_{k=1}^\infty
 \int_{X\times X} \phi_{n,\ell}(x)\phi_{n,\ell}(y) 
 \langle h_k, K(x,y) h_k \rangle_{\sH}
  \nonumber
 \\
 \,\,&=\,\,
 \sum_{\ell=1}^{2^n} 
 \int_{X\times X} \phi_{n,\ell}(x)\phi_{n,\ell}(y) 
\sum_{k=1}^\infty  \langle h_k, K(x,y) h_k \rangle_{\sH}
\label{eq:swapksumwithint}
 \\
\,\,&=\,\,
 \sum_{\ell=1}^{2^n} 
 \int_{X\times X} \phi_{n,\ell}(x)\phi_{n,\ell}(y) 
\operatorname{Tr}(K(x,y))\, dx\,dy
 \nonumber
\\
\,\,&=\,\,
 2^n \sum_{\ell=1}^{2^n} \iint\limits_{\,\,\,\chi_{n,\ell}} \operatorname{Tr}(K(x,y)) 
\,dx \,dy,
 \nonumber
\end{align}
where $\chi_{n,\ell} = \text{Support}(\phi_{n,\ell} \otimes \phi_{n,\ell}).$
On $\chi_{n,m}$,  
  $  K(x,y)\,\, \approx\,\, K\left( \frac{\ell}{2^n}, \frac{\ell}{2^n}\right)$,
and so
\begin{equation}\label{cong1}
\Tr(\mathcal{P}_n \mathcal{K} \mathcal{P}_n) 
\,\,\approx\,\, 
 \sum_{\ell=1}^{2^n} \frac{1}{2^n} \Tr\left(K\left(\frac{\ell}{2^n},\frac{\ell}{2^n}\right)\right)
\,\,\approx\,\,
 \int_X \Tr(K(x,x))\, dx.
\end{equation}
By the uniform continuity of the function $\Tr(K(\cdot,\cdot))$ on the compact set $X\times X$, the error in the first approximation in (\ref{cong1}) converges to 0 as $n \rightarrow \infty.$ Similarly, the error in the second approximation
converges to 0 since  $F(x) = \Tr(K(x,x))$ is continuous on $X$.

To complete the proof, we apply the Lebesgue dominated convergence theorem 
to show that the infinite sum over $k$ can be
taken inside the integral in \eqref{eq:swapksumwithint}.
Specifically, let $\cQ_N : \sH \to \operatorname{Span} \{ h_1,\dots,h_N\}$
be orthogonal projection. Let 
$F_N(x,y) := \sum_{k=1}^N  \langle h_k, K(x,y) h_k \rangle_{\sH} = 
\Tr( \cQ_N K(x,y) \cQ_N)$. Then
$F_N(x,y) \to \Tr(K(x,y))$ pointwise and 
$| F_N(x,y) | \leq \| \cQ_N K(x,y) \cQ_N \|_{\cB_1(\sH)} \leq 
 \| K(x,y) \|_{\cB_1(\sH)}$ by \eqref{eq:TrB1ineq} amd \eqref{eq:SubMult}. Finally, we observe that
 $\| K(x,y) \|_{\cB_1(\sH)} \in L^1(X\times X,\mathbb R)$ as 
 $K\in C^0(X\times X; \cB_1(\sH))$ and $X\times X$ is compact.
\end{proof}

\section{A generalization of Mercer's theorem}\label{sec:Mercer}

Mercer's theorem  states that  a Hilbert-Schmidt operator
on $L^2([a,b];\mathbb C)$ with a continuous, 
non-negative definite Hermitian kernel  is trace class
and that the eigen-expansion of the kernel converges
uniformly~\cite{mercer1909}. 
In this section, we generalize Mercer's theorem to 
the case of Hilbert-Schmidt operators on $L^2(X;\sH)$, where $X$ is compact
and $\sH$ is an infinite-dimensional Hilbert space.
In the case that the Hilbert space, $\sH$, is infinite dimensional we must also assume that
 along the diagonal, $\Delta\subset X\times X$,
  the kernel takes values in $\cB_1(\sH)$, rather than simply in $\cB_2(\sH)$.
 To prove that the operator is trace class we also need to assume that
 the $\cB_1(\sH)$-norm of the 
 restriction of the kernel to the diagonal is integrable (Theorem~\ref{thm:MercerA}).
 To show that the   eigen-expansion of the kernel converges
uniformly we need to impose the more restrictive condition that 
the kernel is bounded in the $\cB_1(\sH)$-norm on $\Delta$ (Theorem~\ref{thm:MercerB}).
Our proof is modeled on the proofs of the classical Mercer's theorem (i.e., $\sH=\mathbb C$)
given by Smithies~\cite{smithies1958} and Cochran~\cite{cochran1972analysis}
We begin with the following elementary result.

\begin{proposition}\label{prop:Cont}
Suppose that $X$ is compact. 
If the kernel is continuous, $K \in C^0(X\times X; \cB_2(\sH))$, then
the eigenfunctions are  continuous, $\psi_\ell,\phi_\ell\in C^{0}(X;\sH)$.
\end{proposition}

\begin{proof}
Since $X$ is compact, $\phi_\ell \in  L^2(X;\sH)\subset L^1(X;\sH)$ and so 
 \begin{equation}
 \| K(x,\cdot)\phi_\ell(\cdot) \|_\sH \leq
  \| K(x,\cdot) \|_{\cB_2(\sH)} \| \phi_\ell(\cdot) \|_\sH
 \leq M \| \phi_\ell(\cdot) \|_\sH
 \end{equation}
  is bounded in  
 $L^1(X;\sH)$ independent of $x\in X$. Therefore, by the 
 Lebesgue dominated convergence theorem, 
 $\lim\limits_{x\to x_0}\psi_\ell(x) = 
 \frac{1}{\mu_\ell} \int_X \lim\limits_{x\to x_0} K(x,y)  \phi_\ell(y)\, dy
 = \psi_\ell(x_0)$, as required.
 Since the adjoint operation is an isometry on $\cB_2(\sH)$, a similar argument shows that $\phi_\ell$ is continuous.
 \end{proof}

\begin{theorem}[Mercer's theorem, Part I]\label{thm:MercerA}
Suppose that $X \subseteq \mathbb R^m$ is compact.
Let $\cP\in\cB_2(L^2(X;\sH))$ be a Hilbert-Schmidt operator on $L^2(X;\sH)$
that is non-negative definite Hermitian, $\cP \geq 0$. 
Suppose that the  kernel of $\cP$
has  a continuous representative, $P\in C^0(X\times X; \cB_2(\sH))$.
In addition, if $\sH$ is infinite dimensional, suppose that
$P(x,x)\in {\cB_1(\sH)}$ for all $x\in X$ and that the
restriction of $P$ to the diagonal, $\Delta \subset X\times X$
is integrable in the  $\cB_1(\sH)$-norm, that is 
$\restr{P}{\Delta} \in L^1(\Delta;\cB_1(\sH))$. 
Then $\cP\in\cB_1(L^2(X;\sH))$  is trace class.
\end{theorem}

\begin{remark}
In Theorem~\ref{thm:MercerB}, under the more restrictive assumption that the
function
$\|\P(x,x)\|_{\cB_1(\sH)}$ is bounded  we prove that 
\begin{equation}\label{eq:MercerTraceFormula}
\Tr(\cP)\,\,=\,\, \int_X \Tr(P(x,x))\, dx.
\end{equation}
\end{remark}

\begin{remark}
The conditions on the Schatten-class properties of the 
kernel can be replaced by the 
stronger condition that $P\in C^0(X\times X; \cB_1(\sH))$.
\end{remark}

The proof of Theorem~\ref{thm:MercerA} relies on the following lemma.

\begin{lemma}\label{lemma:PxxNonNeg}
Let $\cP \in\cB_2(L^2(X;\sH))$ be a non-negative definite Hermitian Hilbert-Schmidt operator on $L^2(X;\sH)$ with kernel $P\in C^0(X\times X;\cB_2(\sH))$.
Then for all $x\in X$, $P(x,x)$ is non-negative definite Hermitian Hilbert-Schmidt 
operator on $\sH$.
\end{lemma}

\begin{proof}
Clearly, $P(x,x)$ is Hermitian. To show that $P(x,x) \geq 0$ we must show that
for all $x\in X$ and $\psi\in\sH$, $\langle P(x,x)\psi,\psi\rangle_\sH \,\,\geq\,\, 0$. 
To obtain a contradiction, suppose instead that $\exists$ $x_0\in X$ and $\psi_0\in \sH$ so that
$\langle P(x_0,x_0)\psi_0,\psi_0\rangle_\sH \,\,<\,\, 0$. Since $P$ is continuous,
there is a ball $B\subset X$ centered at $x_0$ so that 
$\Re \langle P(x,y)\psi_0,\psi_0\rangle_\sH \,\,<\,\, 0$ for all $(x,y)\in B\times B$. 
Let $\phi:X\to\mathbb R$ be a non-negative, 
continuous function with support in $B$ and define $\Psi:X\to\sH$ by
$\Psi(x) = \phi(x)\psi_0$. Then 
\begin{equation}
\langle \mathcal P \Psi\, ,\, \Psi\rangle_{L^2(X;\sH)}
\,\,=\,\, \iint_{B\times B} \phi(x)\phi(y) \,
\Re\langle P(x,y)\psi_0,\psi_0\rangle_\sH
\, dx dy\,\,< \,\, 0,
\end{equation}
contradicting the fact that $\mathcal P \geq 0$. 
\end{proof}

\begin{proof}[Proof of Theorem~\ref{thm:MercerA}]
Let $\{\phi_\ell\}_{\ell=1}^\infty$ be an orthonormal basis of eigenfunctions
of $\cP$ with  eigenvalues $\{\lambda_\ell\}_{\ell=1}^\infty$.
Let $\cP_n = \sum\limits_{\ell=1}^n \lambda_\ell \phi_\ell \otimes \phi_\ell$
and $\cQ_n = \cP - \cP_n$. 
We claim that $\operatorname{Spec}(\cQ_n) \subseteq  \{0\} \cup \{ \lambda_\ell\}_{l=n+1}^\infty$.
Therefore, $\cQ_n \geq 0$, since $\cP\geq 0$.
To verify the claim, we first note that 
$\cQ_n \phi_k = \lambda_k\phi_k$ and  
$\cP_n \phi_k = 0$
for all $k\geq n+1$.
Next suppose that $\lambda \notin  \{ \lambda_\ell\}_{l=n+1}^\infty$ but that
 $\cQ_n\psi = \lambda\psi$, for $\psi\neq 0$. Since $(\lambda-\lambda_\ell) \, \langle\psi,\phi_\ell\rangle_{L^2(X;\sH)} = 0$, we find that $\psi \perp \phi_\ell$ for 
all  $\ell\geq n+1$. But then 
$\lambda\psi = \cQ_n\psi =\sum\limits_{\ell=n+1}^\infty \lambda_\ell \langle\phi_\ell,\psi\rangle_{L^2(X;\sH)}\phi_\ell = 0$, and so $\lambda=0$. 

Next, since we are assuming that  $P(x,x)\in\cB_1(\sH)$, 
we have that $\cQ_n(x,x) = P(x,x) - P_n(x,x)\in\cB_1(\sH)$.
Also, since $\cQ_n \geq 0$, Lemma~\ref{lemma:PxxNonNeg} implies that
 $\cQ_n(x,x) \geq 0$ and so $0\leq \Tr(Q_n(x,x)) = \Tr(P(x,x)) - 
 \Tr(P_n(x,x))$, yielding
 \begin{equation}\label{eq:OperatorPTraceClass1}
\sum\limits_{\ell=1}^n \lambda_\ell \| \phi_\ell(x) \|_\sH^2 \,\,\leq\,\, \Tr(P(x,x)).
\end{equation}
Since 
$|\Tr(P(x,x))| \leq \|P(x,x)\|_{\cB_1(\sH)}$ where 
$\restr{P}{\Delta} \in L^1(\Delta;\cB_1(\sH))$,
the Lebesgue Dominated Convergence Theorem implies that
\begin{equation}\label{eq:OperatorPTraceClass2}
\sum\limits_{\ell=1}^\infty \lambda_\ell 
\,\,\leq \,\, \int_X \Tr(P(x,x))\, dx \,\,< \,\,\infty.
\end{equation}
Therefore, $\cP$ is trace class
since for a  non-negative definite operator the singular values equal
the eigenvalues.
\end{proof}

\begin{remark}
Just as in the classical Mercer's theorem,  
to ensure that $\cP$ is trace class we need to make an assumption
on the smoothness of the kernel. However, even if the kernel
is constant, when $\sH$  is infinite dimensional the following 
example shows that we cannot conclude that  $\cP$ is trace class. 
This example motivates the additional  
assumption that $P(x,x)\in\cB_1(\sH)$ for all $x\in X$.
\end{remark}

\begin{example}\label{ex:MercerCounterExample}
Let $\sH = \ell_2$ be the space of square summable sequences and let
$\{ {\bf e}_n \}_{n=1}^\infty$ be the standard orthonormal basis for $\ell_2$. 
Let $P$ be the Hermitian positive definite 
bounded  operator on $\ell_2$ defined by
\begin{equation}
P\,\,:=\,\, \sum\limits_{n=1}^\infty \frac 1n\, {\bf e}_n \otimes {\bf e}_n.
\end{equation}
Since  the singular values of
$P$ are given by $\mu_n = \frac 1n$, we know that $P\in \cB_2(\ell_2)$
but $P\notin \cB_1(\ell_2)$.
Now let $\mathcal P\in \cB_2(L^2([0,1]\,;\ell_2))$ be the nonnegative definite 
Hilbert-Schmidt operator
with constant operator-valued kernel given by $P(x,y) = P$.
The eigen-expansion for $\cP$ is 
\begin{equation}
\cP\,\,:=\,\, \sum\limits_{n=1}^\infty \frac 1n\, {e}_n \otimes {e}_n,
\label{eq:MercerCounterExample}
\end{equation}
where  ${e}_n\in L^2([0,1]\,;\ell_2)$ is given by ${e_n}(x) := {\bf e}_n$. 
Clearly, $\mathcal P\notin \cB_1(L^2([0,1]\,; \ell_2))$.
We note however that the eigen-expansion for the kernel
of $\mathcal P$ converges uniformly in $\cB_2(\ell_2)$.
\end{example}

\begin{remark}
Continuing Example~\ref{ex:L2Iso},
 let $X\subset \mathbb R^{m_1}$ and $Y\subset \mathbb R^{m_2}$ 
 be compact sets, 
and let $\widetilde\cP \in \cB_2(L^2(X\times Y;\mathbb C))$ be a positive definite
 integral operator with a  continuous, scalar-valued kernel, $\widetilde P$.
 Then by the classical Mercer's theorem, 
 $\widetilde \cP\in \cB_1(L^2(X\times Y;\mathbb C))$. Since the mapping
 $T$ in \eqref{eq:L2Iso1} is an isomorphism,  $\cP=T(\widetilde \cP)
 \in \cB_1(L^2(X;L^2(Y;\mathbb C)))$ is also trace class. 
Alternatively, with the same assumptions on the kernel $\widetilde P$,
we can apply Theorem~\ref{thm:MercerA} to directly show that
$\cP$ is trace class. To check that $\cP$ satisfies the assumptions 
of Theorem~\ref{thm:MercerA}, we first show that 
 $P \in C^0(X\times X;\cB_2(L^2(Y;\mathbb C)))$. 
 To do so, we note that by \eqref{eq:L2Iso2} the kernel for the operator
$P(x,x') \in \cB_2(L^2(Y;\mathbb C))$ is given by
\begin{equation}\label{eq:PtildePkernels}
P(x,x')(y,y') = \widetilde P( (x,y), (x',y'))
\end{equation}
 and, by Theorem~\ref{thm:HSopB2ker},
that 
\begin{equation}
\| P(x_1,x_1') - P(x_2,x_2') \|_{\cB_2(L^2(Y;\mathbb C))} 
= 
\| \widetilde P( (x_1,\cdot), (x_1',\cdot)) - \widetilde P( (x_2,\cdot), (x_2',\cdot)) \|_{L^2(Y\times Y; \mathbb C)}
\end{equation}
which tends to zero as $(x_2,x_2') \to (x_1,x_1')$ since $\widetilde P$ is 
uniformly continuous. 
Second, by \eqref{eq:PtildePkernels} for each $x\in X$
 the kernel for the operator $P(x,x) \in \cB_2(L^2(Y;\mathbb C))$ is continuous.
Moreover, by Lemma~\ref{lemma:PxxNonNeg}, $P(x,x)$ is a nonnegative definite operator.
Therefore by the classical Mercer Theorem 
$P(x,x) \in \cB_1(L^2(Y;\mathbb C))$ is trace class for all $x
\in X$.
Finally, by \eqref{eq:TraceByEval} and Theorem~\ref{thm:IntegralFormulaForTrace}, 
\begin{equation}
\int_X \| P(x,x)\|_{\cB_1(\sH)}\, dx = \int_X \operatorname{Tr}(P(x,x))\, dx
= \int_X \int_Y \widetilde P( (x,y),(x,y'))\, dy dx
< \infty
\end{equation}
since $\widetilde P$ is continuous. 
 
 However, the following example shows that the 
 $\cB_2(L^2(Y;\mathbb C))$-valued case of Theorem~\ref{thm:MercerA} includes examples that cannot be shown to be trace class using the classical theorem.
 
\end{remark}

\begin{example}
 Let $X\subset \mathbb R^{m_1}$ and $Y\subset \mathbb R^{m_2}$ 
 be compact sets. Let $\sH = L^2(Y;\mathbb R)$ 
and let $\{\psi_\ell\}_{\ell=1}^\infty$
be an orthonormal basis for $\sH$ consisting of discontinuous functions.
For example, if $Y=[0,1]$, then we can choose the basis of Haar wavelets~\cite{Burrus}. Let $\{\mu_\ell\}_{\ell=1}^\infty \in \ell_1$
be a summable sequence of positive scalars. Then 
\begin{equation}
P\,\,:=\,\, \sum\limits_{\ell=1}^\infty \mu_\ell \,\psi_\ell \otimes \psi_\ell 
\,\,\in\,\, \cB_1(L^2(Y;\mathbb R)).
\end{equation}
Now let $\cP\in \cB_2(L^2(X;L^2(Y;\mathbb R)))$ be the nonnegative definite 
Hilbert-Schmidt operator
with constant operator-valued kernel 
$P\in L^2(X\times X;\cB_2(L^2(Y;\mathbb R)))$
given by $P(x,x') = P$. Then $\cP$ satisfies the assumptions of 
Theorem~\ref{thm:MercerA} and so is trace class. However, the
kernel $\widetilde P$ for $\widetilde \cP\in\cB_2(L^2(X\times Y;\mathbb R))$ 
is given by
\begin{equation}
\widetilde P( (x,y),(x',y') ) \,\,=\,\, \sum\limits_{\ell=1}^\infty \mu_\ell \,\psi_\ell (y) \psi_\ell(y') 
\end{equation}
which is not continuous. Therefore, we cannot apply the classical Mercer's Theorem to show that $\widetilde \cP$ and hence $\cP$ is trace class. 
\end{example}

The proof of the  general trace class Theorem~\ref{thm:TraceClassInfDimCpt}
only relies on Theorem~\ref{thm:MercerA}. However because it is of independent
interest, we conclude our discussion of Mercer's Theorem by establishing the
uniform convergence of the eigen-expansion for the kernel for 
operator-valued kernels. 

\begin{theorem}[Mercer's theorem, Part II]\label{thm:MercerB}
Suppose that $X \subseteq \mathbb R^m$ is compact.
Let $\cP\in\cB_2(L^2(X;\sH))$ be a Hilbert-Schmidt operator on $L^2(X;\sH)$
that is non-negative definite Hermitian, $\cP \geq 0$. 
Suppose that the  kernel of $\cP$
has  a continuous representative, $P\in C^0(X\times X; \cB_2(\sH))$.
In addition, if $\sH$ is infinite dimensional, suppose 
that $\exists M<\infty$ so that
$\|P(x,x)\|_{\cB_1(\sH)} \leq M$  for all $x\in X$.
Let $\{ \phi_\ell\}_{\ell=1}^\infty$ be an orthonormal basis of 
eigenfunctions  of $\cP$ with eigenvalues $\{ \lambda_\ell\}_{\ell=1}^\infty$.
Then the eigen-expansion,
\begin{equation}\label{eq:Mercer}
P(x,y) \,\,=\,\, \sum\limits_{\ell=1}^\infty
\lambda_\ell \, \phi_\ell(x)\otimes \phi_\ell(y),
\end{equation}
converges uniformly in $\cB_2(\sH)$ in that
\begin{equation}
\sup\limits_{(x,y)\in X\times X} 
\left\| P(x,y) - \sum\limits_{\ell=1}^n \lambda_\ell \phi_\ell(x)\otimes \phi_\ell(y) 
\right\|_{\cB_2(\sH)} \,\,\to\,\, 0 \qquad \text{as } n\to \infty.
\end{equation}
Furthermore,
\begin{equation}\label{eq:MercerTraceFormula}
\Tr(\cP)\,\,=\,\, \int_X \Tr(P(x,x))\, dx.
\end{equation}
\end{theorem}

Before presenting the details, we provide an overview of the proof of the theorem.
\begin{enumerate}
\item[(a)]  The iterated operator
$\cP^2$ is trace class and  it's kernel, $P^{(2)}$, is continuous.
(Proposition~\ref{prop:P2xxTrace})
\item[(b)]   $\sum\limits_{\ell=1}^\infty
\lambda_\ell^2 \, \| \phi_\ell(x) \|_\sH^2$ converges pointwise
since it is bounded  by the partial integral with respect to $y$ 
of $\| P(x,y)\|^2_{\cB_2(\sH)}$. (Proposition~\ref{prop:P2TraceBound})
\item[(c)]  The eigen-expansion 
$P^{(2)}(x,y) = 
\sum\limits_{\ell=1}^\infty \lambda_\ell^2 \, \phi_\ell(x)\otimes \phi_\ell(y)$
converges pointwise in $\cB_2(\sH)$ thanks to the Cauchy-Schwarz inequality
and  (b). (Proposition~\ref{prop:P2PtWiseConvgt})
\item[(d)]  
$\Tr(P^{(2)}(x,x)) = 
\sum\limits_{\ell=1}^\infty \lambda_\ell^2 \, \|\phi_\ell(x) \|_\sH^2$
converges uniformly. (Proposition~\ref{prop:P2UnifConvgtTrace}) 
In the case that $\dim(\sH)$ is infinite, this result relies on the fact
that for any Hermitian matrix, $\mathbf A$, and any vector $\mathbf v$,
$\Tr(\mathbf A) \leq \Tr(\mathbf A + \mathbf v\mathbf v^*)$, since
the eigenvalues of $\mathbf A + \mathbf v\mathbf v^*$ are interleaved
with those of $\mathbf A$. (Proposition~\ref{prop:AbstractTraceFormula} and Lemma~\ref{lemma:Interleaved})
\item[(e)] The eigen-expansion 
$P(x,y) = 
\sum\limits_{\ell=1}^\infty \lambda_\ell \, \phi_\ell(x)\otimes \phi_\ell(y)$
converges uniformly in $\cB_2(\sH)$  thanks to the Cauchy-Schwarz inequality
and  (d). 
It is in this step that the more restrictive condition that
$x\mapsto \| P(x,x) \|_{\cB_1(\sH)}$ is bounded is required.
(Lemma~\ref{lemma:PeqR} and Proposition~\ref{prop:PkernelUnifCauchy})
\item[(f)] Finally, we use Proposition~\ref{prop:AbstractTraceFormula} to prove \eqref{eq:MercerTraceFormula} for the trace of $\cP$.
\end{enumerate}

We now state and prove these results.
First we note that by \eqref{eq:SubMult}, the composite operator, $\cP^2=\cP\circ\cP$ is trace class.
We begin with two propositions about the kernel,  $P^{(2)}$, of $\cP^2$.

\begin{proposition} \label{prop:P2xxTrace}
Let $X \subseteq \mathbb R^m$ be compact
and let  $\cP \in \cB_2(L^2(X;\sH))$ be Hermitian. Furthermore, 
suppose that $\cP$ has a  continuous kernel, $P\in C^0(X\times X; \cB_2(\sH))$. 
Then the kernel, $P^{(2)}$,  of $\cP^2$ is continuous and takes values in
 $\cB_1(\sH)$, that is, 
$P^{(2)} \in C^0(X\times X; \cB_1(\sH))$.
\end{proposition}

\begin{proof}
By Remark \ref{rem:compositionHS}, \eqref{eq:TriInt} and \eqref{eq:SubMult}, for each $x,y\in X$,
\begin{align}
\| P^{(2)}(x,y) \|_{\cB_1(\sH)}
 \,\,& \leq \,\, \int_X \| P(x,z) P(z,y) \|_{\cB_1(\sH)} \, dz
\\
 \,\,& \leq \,\, \int_X \| P(x,z)\|_{\cB_2(\sH)} \| P(z,y) \|_{\cB_2(\sH)} \, dz
 \\
 \,\, & \leq \,\, \left( \int_X \| P(x,z)\|_{\cB_2(\sH)}^2 \, dz\right)^{1/2} \,
 \left( \int_X \| P(z,y)\|_{\cB_2(\sH)}^2 \, dz\right)^{1/2},
\end{align}
which is finite since $P\in C^0(X\times X; \cB_2(\sH))$ and $X$ is compact.
Therefore for all $x,y\in X$,  $P^{(2)}(x,y) \in \cB_1(\sH)$.
A similar calculation shows that
\begin{align}
&\| P^{(2)}(x_1,y_1) -  P^{(2)}(x_2,y_2) \|_{\cB_1(\sH)}
\nonumber
\\
&\,\,\leq\,\,
\left( \int_X \| P(x_1,z)\|_{\cB_2(\sH)}^2 \, dz\right)^{1/2}\,
\left( \int_X \| P(z,y_1) -  P(z,y_2) \|_{\cB_2(\sH)}^2 \, dz\right)^{1/2}
\nonumber
\\
&\,\,+\,\,
\left( \int_X \| P(z,y_2)\|_{\cB_2(\sH)}^2 \, dz\right)^{1/2}\,
\left( \int_X \| P(x_1,z) -  P(x_2,z) \|_{\cB_2(\sH)}^2 \, dz\right)^{1/2}
\end{align}
which implies that the kernel $P^{(2)}$ is continuous since $P$ is
continuous and $X$ is compact.
\end{proof}

\begin{proposition} \label{prop:P2TraceBound}
Let $X \subseteq \mathbb R^m$ be compact
and let  $\cP \in \cB_2(L^2(X;\sH))$ be Hermitian with continuous kernel
$P\in C^0(X\times X; \cB_2(\sH))$. 
Let $\{(\lambda_\ell, \phi_\ell)\}_\ell$ be the eigenpairs for $\cP$. Then 
for all $n$, 
\begin{equation}\label{eq:P2TraceBound}
\sum\limits_{\ell=1}^n \lambda_\ell^2 \| \phi_\ell(x) \|^2
\,\,\leq\,\,
\| P_x\|^2_{L^2(X;\cB_2(\sH))},
\end{equation}
where
$\| P_x\|^2_{L^2(X;\cB_2(\sH))} := \int_X \| P(x,y) \|^2_{\cB_2(\sH)}\, dy 
< \infty$.
Furthermore,  
$\sum\limits_{\ell=1}^\infty \lambda_\ell^2 \| \phi_\ell(x) \|^2$
converges pointwise in $x$. 
\end{proposition}

\begin{proof}
Let $\{ h_k\}_k$ be an orthonormal basis for $\sH$ and $P_{km}(x,y) = 
\langle P(x,y)h_k,h_m\rangle_\sH$. Then, using $\left\| \sum\limits_{k=1}^\infty \alpha_k h_k \right\|_\sH^2 =  \sum\limits_{k=1}^\infty |\alpha_k|^2$, we have that
\begin{align}
\sum\limits_{\ell=1}^n \lambda_\ell^2 \left\| \phi_\ell(x) \right\|_\sH^2
\,\,&=\,\,
\sum\limits_{\ell=1}^n \| (\cP\phi_\ell)(x) \|_\sH^2
\nonumber
\\
\,\,&=\,\,
\sum\limits_{\ell=1}^n \,\left\| \, \int_X P(x,y)  \sum\limits_{k=1}^\infty\,
\langle h_k ,\phi_\ell(y)\rangle_\sH \,h_k\, dy \right\|_\sH^2
\nonumber
\\
\,\,&=\,\,
\sum\limits_{\ell=1}^n \,\left\|  \,\int_X 
 \sum\limits_{k=1}^\infty \,\langle h_k, \phi_\ell(y)\rangle_\sH\,
\sum\limits_{m=1}^\infty P_{km}(x,y)h_m\, dy \right\|_\sH^2
\nonumber
\\
\,\,&=\,\,
\sum\limits_{\ell=1}^n \,\left\| \, \sum\limits_{m=1}^\infty  \left( 
\int_X  \sum\limits_{k=1}^\infty P_{km}(x,y) \,
\langle h_k, \phi_\ell(y)\rangle_\sH\, dy\right)
h_m\right\|_\sH^2
\nonumber
\\
\,\,&=\,\,
  \sum\limits_{\ell=1}^n \sum\limits_{m=1}^\infty \,\left|   
\int_X \,\left\langle \sum\limits_{k=1}^\infty P_{km}(x,y)  h_k,
  \phi_\ell(y)\right\rangle_\sH\, dy \,\right|^2
  \nonumber
\\
\,\,&=\,\,
\sum\limits_{m=1}^\infty
  \sum\limits_{\ell=1}^n  \,\left|   
 \,\left\langle \sum\limits_{k=1}^\infty P_{km}(x,\cdot)  h_k,
  \phi_\ell(\cdot)\right\rangle_{L^2(X;\sH)} \,\right|^2
    \nonumber
\\
\,\,&\leq\,\,
\sum\limits_{m=1}^\infty  \, \int_X\, 
\left\| \sum\limits_{k=1}^\infty P_{km}(x,y)  h_k \right\|_\sH^2\, dy
 \nonumber
 \\
\,\,&=\,\,
\int_X \sum\limits_{k,m=1}^\infty | P_{km}(x,y) |^2\, dy
\,\,=\,\, 
\int_X \, \|P(x,y)\|_{\cB_2(\sH)}^2\,dy,
\end{align}
which is finite since $P$ is continuous and $X$ is compact, 
and where the inequality follows by Bessel's inequality.
Finally, since the sequence of partial sums in \eqref{eq:P2TraceBound}
is increasing and bounded above, the series converges pointwise.
\end{proof}

The following definition is due to Smithies~\cite{smithies1958}.

\begin{definition}
We say that a sequence $f_n \in L^2(X;\sH)$ is 
\emph{relatively uniformly convergent} to $f \in L^2(X;\sH)$ if there is a
function $g\in L^2(X;\sH)$ so that for all $\epsilon > 0$ there is an 
$N=N(\epsilon)$
so that for all  $n > N$ and $x\in X$ 
\begin{equation}
\| f_n(x) - f(x) \|_\sH \,\, < \,\, g(x).
\end{equation}
\end{definition}

\begin{remark}
If $f_n$ is relatively uniformly Cauchy (in the obvious sense),
then there exists an $f$ so that
$f_n$ is relatively uniformly convergent to $f$. Also,
if $f_n$ is relatively uniformly convergent to $f$, then $f_n$ converges pointwise
to $f$.
\end{remark}

\begin{proposition} \label{prop:P2PtWiseConvgt}
Let $X \subseteq \mathbb R^m$ be compact
and let  $\cP \in \cB_2(L^2(X;\sH))$ be Hermitian with continuous kernel
$P\in C^0(X\times X; \cB_2(\sH))$. 
Let $P^{(2)}$ be the kernel associated with the  operator $\cP^2$.
Then,
\begin{equation}\label{eq:P2series}
P^{(2)}(x,y) \,\,=\,\, \sum\limits_{\ell=1}^\infty
\lambda^2_\ell \, \phi_\ell(x)\otimes \phi_\ell(y) 
\qquad\text{for all } (x,y)\in X\times X,
\end{equation}
where the series is pointwise convergent  in $\cB_2(\sH)$.
\end{proposition}

\begin{proof}
We begin by showing that, for each fixed $x$, 
the series on the right-hand side of \eqref{eq:P2series}
is relatively uniformly Cauchy in $y$. Indeed, by the Cauchy-Schwarz inequality,
\eqref{normab}, and Proposition~\ref{prop:P2TraceBound}, we have that
\begin{align}
\left\| \sum\limits_{\ell=n}^m 
\lambda_\ell^2 \phi_\ell(x) \otimes \phi_\ell(y) \right\|_{\cB_2(\sH)}
\,\,&\leq\,\,
\left[ \sum\limits_{\ell=n}^m  \lambda_\ell^2 \|\phi_\ell(x)\|_\sH^2 \right]^{1/2}
\left[ \sum\limits_{\ell=n}^m  \lambda_\ell^2 \|\phi_\ell(y)\|_\sH^2 \right]^{1/2}
\nonumber
\\
\,\,&\leq\,\,
\epsilon \,\| P_y \|_{L^2(X;\cB_2(\sH))}.
\end{align}
Since $P$ is  bounded, for each fixed $x$, the series is  uniformly Cauchy in $y$,
and hence converges uniformly in $y$ to a kernel $\widetilde{P}^{(2)}(x,y)$.
Furthermore, since the eigenfunctions are continuous in $y$, for each $x$,
the limit $\widetilde{P}^{(2)}(x,y)$ is continuous in $y$.
Similarly, for each $y$,
$\widetilde{P}^{(2)}(x,y)$ is continuous in $x$.
Hence $\widetilde{P}^{(2)}\in C^0(X\times X;\cB_2(\sH))$.
Let $Q_n(x,y)$ be the $n$-th partial sum of the series in \eqref{eq:P2series}.
We know that $Q_n \to \widetilde{P}^2$
pointwise. Since $Q_n(x,y) \leq \| P_x \|_{L^2(X;\cB_2(\sH))} \| P_y \|_{L^2(X;\cB_2(\sH))}$, $Q_n$ is bounded  and so 
$Q_n \to \widetilde{P}^2$ in $L^2(X\times X;\cB_2(\sH))$.
Also, by Theorem~\ref{thm:KernelExpansion}, $Q_n \to P^2$ in $L^2(X\times X;\cB_2(\sH))$. 
So by uniqueness of limits, $ \widetilde{P}^2(x,y) = P^2(x,y)$ 
for almost all $(x,y)$. Since both kernels are continuous, they are equal
everywhere.
\end{proof}

\begin{remark}
Since the kernel, $P^{(2)}$ takes values in $\cB_1(\sH)$, it is tempting
to conclude from \eqref{eq:P2series}
 that $\Tr(P^{(2)}(x,x)) \,\,=\,\, \sum\limits_{\ell=1}^\infty
\lambda^2_\ell \, \| \phi_\ell(x) \|^2_\sH$. However, this is not guaranteed
since we do not  know that the series in \eqref{eq:P2series} converges 
pointwise in $\cB_1(\sH)$,
or if it does that  it actually converges to $\Tr(P^{(2)}(x,x))$. 
One of the challenges in establishing such a result is that for 
a given $x$, the sequence $\{\phi_\ell(x)\}_\ell \in \sH$ does not form a 
basis for $\sH$, let alone an orthonormal one. 
To address this issue, we appeal to the following result.
\end{remark}

\begin{proposition}\label{prop:AbstractTraceFormula}
Let $\{\chi_\ell\}_\ell$ be any sequence of elements of $\sH$ for which
\begin{equation}
\cA\,\,=\,\, \sum\limits_{\ell=1}^\infty \chi_\ell \otimes \chi_\ell
\end{equation}
converges in $\cB_2(\sH)$ and suppose that $\cA \in \cB_1(\sH)$. Then
\begin{equation}\label{eq:AbstractTraceFormula}
\Tr(\cA)\,\,=\,\, \sum\limits_{\ell=1}^\infty \| \chi_\ell \|_{\sH}^2.
\end{equation}
\end{proposition}

\begin{proof}
The operators
\begin{equation}
\cA_n\,\,=\,\, \sum\limits_{\ell=1}^n \chi_\ell \otimes \chi_\ell
\end{equation}
and $\cD_n = \cA-\cA_n$ are trace class, and since $\cD_n \geq 0$
we have that 
\begin{equation}
\sum\limits_{\ell=1}^n \| \chi_\ell \|_{\sH}^2 \,\,\leq \,\, \Tr(\cA).
\end{equation}
Therefore the series on the right hand side of \eqref{eq:AbstractTraceFormula}
converges. 

Let $\lambda_1^{(n)} \geq \lambda_2^{(n)} \geq \cdots 
\lambda_k^{(n)} \geq \cdots  \geq 0$ be the eigenvalues of the operator
$\cA_n$,  and 
$\lambda_1 \geq \lambda_2 \geq \cdots 
\lambda_k\geq \cdots  \geq 0$ the eigenvalues of $\cA$,
both enumerated with algebraic multiplicity.
By \cite[Theorem I.4.2]{gohberg1969introduction}, since $\cA_n\to\cA$ in 
$\cB_\infty(\sH)$, we have that $\lambda_k^{(n)} \to \lambda_k$, as $n\to\infty$.

Furthermore, by Lemma~\ref{lemma:Interleaved} below, $0\leq \lambda_k^{(n)} \leq \lambda_k^{(n+1)}$
for all $k,n$. Therefore by the monotone convergence theorem for series
\begin{equation}
\lim\limits_{n\to\infty} \,\sum\limits_{k=1}^\infty  \lambda_k^{(n)}
\,\,=\,\, \sum\limits_{k=1}^\infty  \lambda_k.
\end{equation}
Finally, since $\cA$ is trace class,
\begin{align}
\Tr(\cA) 
\,\,=\,\, \sum\limits_{k=1}^\infty  \lambda_k
\,\,=\,\,\lim\limits_{n\to\infty} \,\sum\limits_{k=1}^\infty  \lambda_k^{(n)}
\,\,=\,\, \lim\limits_{n\to\infty} \Tr(\cA_n)
\,\,=\,\, \sum\limits_{\ell=1}^\infty \| \chi_\ell \|_{\sH}^2, 
\end{align}
as required.
\end{proof}

\begin{lemma}\label{lemma:Interleaved}
Let $\{\chi_\ell\}_{\ell=1}^{n}$ be any sequence of elements of $\sH$ 
and $\cA = \sum\limits_{\ell=1}^n \chi_\ell\otimes \chi_\ell$.
Then, for any $\chi\in\sH$,
the eigenvalues, $\mu_k$, of $\cA + \chi\otimes\chi$ 
are interleaved with the eigenvalues, $\lambda_k$, of $\cA$:
 \begin{equation}\label{eq:Interleaved}
0\leq \cdots \leq \lambda_{k+1} \leq \mu_{k+1}
\leq\lambda_k\leq\mu_k \leq \cdots \leq \lambda_1 \leq \mu_1.
\end{equation}
 \end{lemma}

The proof is an extension of Exercise 7.1.22 of Meyer~\cite{Meyer}.

\begin{proof}
Let $\sF = \operatorname{Span}\{ \chi_1, \cdots \chi_n, \chi\}$ and let
$m=\dim(\sF)$.
Observe that $\cA : \sF \to \sF$ and  $\cA : \sF^\perp \to \{ 0\}$.
Let $\{f_j\}_j$ be an orthonormal basis of $\sF$ consisting of eigenvectors
of $\cA$ and let $\nu_1 > \nu_2 > \cdots > \nu_L$ denote the distinct eigenvalues
of $\restr{\cA}{\sF}$ and let $m_1, m_2, \cdots, m_L$ be their algebraic 
multiplicities. Then the matrix of $\restr{\cA}{\sF}$  in the basis $\{f_j\}_j$  is 
given by
$\mathbf A = \operatorname{diag}\begin{bmatrix} 
  \nu_1 \mathbf I_{m_1\times m_1} & 
  \nu_2 \mathbf I_{m_1\times m_2} & 
  \cdots &
   \nu_L \mathbf I_{m_1\times m_L} 
   \end{bmatrix}$
and we express the coordinate vector of $\chi$ in this basis
as 
$\mathbf v = \begin{bmatrix} 
  \mathbf v_1^T & 
   \mathbf v_2^T & 
  \cdots &
     \mathbf v_L^T
       \end{bmatrix}^T$
where $\mathbf v_j \in \mathbb C^{m_j}$.
Note that if $\mathbf v_j = \mathbf 0$, then $\nu_j$ is also an eigenvalue of
$\mathbf A+\mathbf v\mathbf v^*$ with the same multiplicity, $m_j$.
Therefore, we may assume that $\mathbf v_j\neq \mathbf 0$ for all $j$ and
it suffices to show that the eigenvalues of $\mathbf A$ and 
$\mathbf A + \mathbf v\mathbf v^*$ are interleaved as in  \eqref{eq:Interleaved}.

To that end, we claim that
\begin{enumerate}
\item If $m_k=1$, then $\nu_k$ is not an eigenvalue of 
$\mathbf A + \mathbf v\mathbf v^*$.
\item If $m_k>1$, then $\nu_k$ is an eigenvalue of 
$\mathbf A + \mathbf v\mathbf v^*$ with multiplicity $m_k-1$. 
\item If $\mu$ is an eigenvalue of 
$\mathbf A + \mathbf v\mathbf v^*$ but not an eigenvalue of 
$\mathbf A$, then
\begin{equation}\label{eq:DiniMuEq}
0 \,\,=\,\, f(\mu) \,\,:=\,\, 1\,\,+\,\, \sum_{k=1}^n \frac{\|\mathbf v_k \|^2}{\nu_k-\mu}.
\end{equation}
\end{enumerate}
To establish these claims, we observe that 
\begin{align}
p(\nu) \,\,&:=\,\, \det(\mathbf A +  \mathbf v\mathbf v^* - \nu\mathbf I)
\,\,=\,\, \det(\mathbf A - \nu\mathbf I) \left[ 1 + \mathbf v^* ( \mathbf A - \nu\mathbf I)^{-1} \mathbf v\right]
\nonumber
\\ \label{eq:DiniClaim3Proof}
&=\,\, \prod_{\ell=1}^L (\nu_\ell-\nu)^{m_\ell} \left[ 1 + \sum_{\ell=1}^L 
\frac{\|\mathbf v_\ell \|^2}{\nu_\ell-\nu} \right].
\end{align}

Fixing $k$, we have 
$p(\nu)= \left[\prod_{\ell\neq k} (\nu_\ell-\nu)^{m_\ell}\right]
( \nu_k-\nu)^{m_k-1} q_k(\nu) $, 
where
\begin{equation}
q_k(\nu) \,\,=\,\, \nu_k - \nu + \| \mathbf v_k  \|^2 
+ \sum_{\ell\neq k} \| \mathbf v_\ell  \|^2  
\,\frac{ \nu_k-\nu}{\nu_\ell-\nu}.
\end{equation}
Claims (1) and (2) now follow, since $q_k(\nu_k)\neq 0$,
and claim (3) follows from \eqref{eq:DiniClaim3Proof}.
Finally, to show that the remaining eigenvalues of $\mathbf A + \mathbf v\mathbf v^*$ are interleaved with those of $\mathbf A$, we observe that
$f: (\nu_{k+1}, \nu_{k}) \to (-\infty, +\infty)$ is 
continuous and increasing. Hence $f$ has  a unique root in $(\nu_{k+1}, \nu_{k})$. A similar conclusion holds on the interval $(\nu_1,+\infty)$.
\end{proof}

We are now in  a position to prove the key result that will be used to 
establish the uniformity of the
 convergence of the eigen-expansion in Mercer's theorem.
 
 \begin{proposition}\label{prop:P2UnifConvgtTrace}
 Let $X \subseteq \mathbb R^m$ be compact
and let  $\cP \in \cB_2(L^2(X;\sH))$ be Hermitian with continuous kernel
$P\in C^0(X\times X; \cB_2(\sH))$. Then
 \begin{equation}\label{eq:P2UnifConvgtTrace}
 \Tr(P^{(2)}(x,x)) \,\,=\,\, \sum\limits_{\ell=1}^\infty
\lambda^2_\ell \, \| \phi_\ell(x) \|^2_\sH,
\end{equation}
where the series on the right-hand side converges uniformly in $x$. 
 \end{proposition}

\begin{proof}
 By Propositions~\ref{prop:P2xxTrace}, \ref{prop:P2PtWiseConvgt}, and
 \ref{prop:AbstractTraceFormula}, \eqref{eq:P2UnifConvgtTrace} holds pointwise
 in $x$. Now the partial sums form an increasing sequence of continuous real-valued functions on the compact set, $X$. Furthermore, by 
 Proposition~\ref{prop:P2xxTrace} and the continuity of the functional
 $\Tr: \cB_1(\sH)\to \mathbb R$, the function $ \Tr(P^{(2)}(x,x))$ is 
 continuous. Therefore, the convergence is uniform by Dini's theorem~\cite{rudin1976}.
\end{proof}
 
 \begin{lemma}\label{lemma:PeqR}
 Let $X$ be compact
and suppose that $\cP \geq 0$ has a  
continuous kernel $P\in C^0(X\times X; \cB_2(\sH))$. 
Let $\{(\lambda_\ell, \phi_\ell)\}_\ell$ be the eigenpairs of $\cP$.
If 
\begin{equation}\label{eq:PeqR}
R(x,y) \,\,=\,\, \sum\limits_{\ell=1}^\infty \lambda_\ell \,
\phi_\ell(x)  \otimes \phi_\ell(y)
\end{equation}
converges uniformly in $\cB_2(\sH)$ in either $x$ or $y$, then $P(x,y)=R(x,y)$.
 \end{lemma}

\begin{proof}
First, let $Q\in C^0(X\times X; \cB_2(\sH))$ be any Hermitian kernel. Let
$\{ h_k\}_k$ be an orthonormal basis for $\sH$. Then by \eqref{eq:HSnormViaONB},
the standard monotone convergence theorem, Remarks~\ref{rem:AdjointHS}
and \ref{rem:compositionHS}, \eqref{eq:SwapIntInnerProd} and 
\eqref{eq:DefTrace},
\begin{align}
\int_X \| Q(x,y) \|^2_{\cB_2(\sH)}\, dx 
\,\,&=\,\,
\sum\limits_{k=1}^\infty \int_X \| Q(x,y)\, h_k \|^2_{\sH}\, dx 
\nonumber
\\
\,\,&=\,\,
\sum\limits_{k=1}^\infty \int_X \langle Q(y,x)\circ Q(x,y)\, h_k, h_k \rangle_{\sH}\, dx 
\nonumber
\\
\,\,&=\,\,
\sum\limits_{k=1}^\infty  \langle Q^{(2)}(y,y)\, h_k, h_k \rangle_{\sH}
\,\,=\,\, 
\operatorname{Tr}(Q^{(2)}(y,y)).
\end{align}
Similarly, $\int_X \| Q(x,y) \|^2_{\cB_2(\sH)}\, dy =  \| Q^{(2)}(x,x)\|^2_{\cB_2(\sH)}$.
Consequently, 
\begin{equation}
\int_X \| \,P(x,y) \,\,-\,\, \sum\limits_{\ell=1}^n
\lambda_\ell\, \phi_\ell(x) \otimes \phi_\ell(y)\, \|^2_{\cB_2(\sH)}\, dy 
\,\,=\,\,
\Tr(P^{(2)}(x,x)) -  \sum\limits_{\ell=1}^n \lambda_\ell^2 \| \phi_\ell(x)\|_\sH^2
\end{equation}
converges to zero uniformly in $x$ as $n\to\infty$, by Proposition~\ref{prop:P2UnifConvgtTrace}.
Therefore, since \eqref{eq:PeqR} converges uniformly in $y$, 
 the uniform convergence theorem for integrals implies that
\begin{equation}
\int_X \| \,P(x,y) - R(x,y) \, \|^2_{\cB_2(\sH)}\, dy = 0
\end{equation}
for all $x$. Consequently, $P(x,y) = R(x,y)$ for almost all $y$. 
Finally, since $P$ and $R$ are both continuous, $P(x,y)=R(x,y)$ for all $x,y$.
\end{proof}

The following result shows that the eigen-expansion for the kernel $P(x,y)$
converges uniformly.

\begin{proposition}\label{prop:PkernelUnifCauchy}
Let $X$ be compact
and suppose that $\cP \geq 0$ has a  
continuous kernel $P\in C^0(X\times X; \cB_2(\sH))$
and that $\exists M<\infty$ so that for all $x\in X$
$\|P(x,x)\|_{\cB_1(\sH)} \leq M$.
Then for each $x\in X$, 
\begin{equation}\label{eq:PkernelUnifCauchy}
P_n(x,y)\,\,=\,\, \sum\limits_{\ell=1}^n \lambda_\ell \,\phi_\ell(x) \otimes \phi_\ell(y)
\end{equation}
is uniformly Cauchy with respect to $y$ in the $\cB_2(\sH)$-norm and hence
converges uniformly to $P(x,y)$. 
\end{proposition}

\begin{proof}
Let $\{h_k\}_k$ be an orthonormal basis for $\sH$ and $n>m$. Then,
\begin{align}
& \| P_n(x,y) - P_m(x,y) \|_{\cB_2(\sH)}^2 
\\
\,\,&=\,\,
\sum\limits_{k=1}^\infty \| (P_n(x,y) - P_m(x,y))h_k \|_{\sH}^2 
\nonumber
\\
\,\,&=\,\,
\sum\limits_{k=1}^\infty \left\| \sum\limits_{\ell=m+1}^n \lambda_\ell \,
\langle \phi_\ell(y),h_k\rangle_\sH\, \phi_\ell(x) \right \|_\sH^2
\nonumber
\\
\,\,&\leq\,\,
\sum\limits_{k=1}^\infty \left( \sum\limits_{\ell=m+1}^n  \lambda_\ell 
| \langle \phi_\ell(y),h_k\rangle_\sH |^2\right) \, 
\left( \sum\limits_{\ell=m+1}^n \lambda_\ell  \| \phi_\ell(x) \|_\sH^2 \right)
\nonumber
\\
\,\,&=\,\,
 \left( \sum\limits_{\ell=m+1}^n  \lambda_\ell \sum\limits_{k=1}^\infty
| \langle \phi_\ell(y),h_k\rangle_\sH |^2\right) \, 
\left( \sum\limits_{\ell=m+1}^n \lambda_\ell  \| \phi_\ell(x) \|_\sH^2 \right)
\nonumber
\\
\,\,&=\,\,
 \left( \sum\limits_{\ell=m+1}^n  \lambda_\ell 
  \| \phi_\ell(y) \|_\sH^2 \right)
\left( \sum\limits_{\ell=m+1}^n \lambda_\ell  \| \phi_\ell(x) \|_\sH^2 \right)
\nonumber
\\
\,\,&\leq\,\,
\left( \sum\limits_{\ell=m+1}^n  \lambda_\ell 
  \| \phi_\ell(x) \|_\sH^2 \right)  \Tr(P(y,y)),
\end{align}
where the first inequality follows from the Cauchy-Schwarz inequality
for $a_\ell\in\mathbb C$ and $\phi_\ell\in\sH$ given by
\begin{equation}
\left\| \,\sum\limits_{\ell=1}^n a_\ell \phi_\ell \,\right\|_\sH 
\,\,\leq\,\,
\left( \sum\limits_{\ell=1}^n |a_\ell |^2\right)^{1/2}\,
\left( \sum\limits_{\ell=1}^n \|\phi_\ell |_\sH^2\right)^{1/2},
\end{equation}
and the second inequality follows from~\eqref{eq:OperatorPTraceClass1}.

We note that $\sum\limits_{\ell=1}^\infty \lambda_\ell \| \phi_\ell \|_\sH^2 < \infty$
for each $x$ by \eqref{eq:OperatorPTraceClass1}. 
Fixing $x\in X$, let $\epsilon>0$ and choose $N=N(\epsilon,x)$ so that
$\sum\limits_{\ell=m+1}^n  \lambda_\ell  \| \phi_\ell(x) \|_\sH^2  < \epsilon$.
Since $\Tr(P(y,y)) \leq M$ is uniformly bounded by assumption, 
we conclude that
$\| P_n(x,y) - P_m(x,y) \|_{\cB_2(\sH)} \leq (\epsilon M)^{1/2}$,
which proves that \eqref{eq:PkernelUnifCauchy} is uniformly Cauchy in $y$
and hence converges uniformly to some kernel, $R(x,y)$.
The fact that $R(x,y)=P(x,y)$ now follows by Lemma~\ref{lemma:PeqR}.
\end{proof}

To complete the proof of Mercer's theorem, we show that~\eqref{eq:MercerTraceFormula} holds.

\begin{proof}[Completion of proof of Mercer's Theorem~\ref{thm:MercerB}]
Since $P(x,x)\in\cB_1(\sH)$ for all $x$ and since the eigen-expansion \eqref{eq:Mercer} for $P(x,x)$ converges pointwise in $\cB_2(\sH)$, Proposition~\ref{prop:AbstractTraceFormula} implies that
$\Tr(P(x,x)) = \sum\limits_{\ell=1}^\infty \lambda_\ell\, \|\phi_\ell(x)\|_\sH^2$.
Finally, since $|\Tr(P(x,x))| \leq \| P(x,x) \|_{\cB_1(\sH)}$, the 
Dominated Convergence Theorem yields
$\int_X  \Tr(P(x,x))\, dx = \sum\limits_{\ell=1}^\infty  \lambda_\ell
= \Tr(\cP)$, as required.
\end{proof}

\section{The case of general operator-valued kernels on a compact set}
\label{sec:OperatorValued}

In this section we treat the case of an arbitrary kernel, that is we drop
the assumption that the operator is nonnegative definite Hermitian.  
We begin with a definition and the statement of the main theorem.

\begin{definition}
We let $C^{0,\gamma}(X\times X; \cB_2(\sH))\subset L^2(X\times X; \cB_2(\sH))$ denote the space
of operator-valued  H\"older continuous kernels, that is kernels, $K$, for which 
there is a  constant $C>0$, so that for
all $x_{1}, x_2,y_1,y_2\in X$,
\begin{equation}\label{eq:holdconddef}
\|K(x_1,y_1)-K(x_2,y_2)\|_{\cB_2(\sH)}\le C\|(x_1,y_1)-(x_2,y_2)\|_{\bbR^{2m}}^\gamma.
\end{equation}
\end{definition}

\begin{theorem}\label{thm:TraceClassInfDimCpt}
Suppose that $X\subseteq \mathbb R^m$ is compact and that 
$K\in C^{0,\gamma}(X\times X; \cB_2(\sH))$ is an operator-valued 
H\"older continuous kernel with  H\"older exponent, $\gamma\in(1/2,1]$.
Let $\cK\in\cB_2(L^2(X;\sH))$ be the 
Hilbert-Schmidt operator associated with $K$. 
In addition, if $\dim\sH = \infty$ suppose
  that  the kernel of $\cP := (\cK\cK^*)^{1/2}$ satisfies
\begin{equation}\label{eq:TraceClassInfDimCptCond}
P \in L^\infty(X\times X;\cB_1(\sH)).
\end{equation}
Then
$\cK\in\cB_1(L^2(X;\sH))$ is trace class.
\end{theorem}

The proof of the theorem, which is based on the proof for scalar-valued kernels
given by Weidmann~\cite{weidmann1966integraloperatoren}, proceeds 
via the following results. 

First, by Corollary~\ref{cor:KernelExpansion}, the eigen-expansion 
\begin{equation}\label{eq:TCKkernelExpansion}
K(x,y) \,\,=\,\, \sum\limits_{\ell=1}^\infty \mu_\ell \,\psi_\ell(x) \otimes \phi_\ell(y)
\end{equation}
for the kernel of $\cK$
converges in $L^2(X\times X;\cB_2(\sH))$.
Formally then, the kernel for the operator, $\cP$, is given by
\begin{equation}\label{eq:TCPkernelExpansion}
P(x,y) \,\,=\,\, \sum\limits_{\ell=1}^\infty \mu_\ell \,\psi_\ell(x) \otimes \psi_\ell(y).
\end{equation}

\begin{proposition}
The eigen-expansion~\eqref{eq:TCPkernelExpansion} converges in 
$L^2(X\times X;\cB_2(\sH))$ and $P$ is the kernel for $\cP$.
\end{proposition}

\begin{proof}
Let $P_n$ be the $n$-th partial sum of the series in~\eqref{eq:TCPkernelExpansion}. Then, letting $\{h_j\}$ be an orthonormal
basis for $\sH$, by \eqref{eq:HSInnerProductViaONB} we have that
\begin{align}
&\| P_n-P_m \|^2_{L^2(X\times X;\cB_2(\sH))} 
\nonumber \\
&\,\,=\,\,
\sum\limits_{k,\ell=m+1}^{n}\,\, \iint\limits_{X\times X}\langle \mu_\ell \psi_\ell(x) \otimes \psi_\ell(y)\, ,\, \mu_k \psi_k(x) \otimes \psi_k(y) \rangle_{\cB_2(\sH)}\, dxdy
\nonumber \\
&\,\,=\,\,
\sum\limits_{k,\ell=m+1}^{n}\,\, \iint\limits_{X\times X}
\mu_\ell  \mu_k
\sum\limits_j 
\langle  (\psi_\ell(x) \otimes \psi_\ell(y))h_j,\, 
(\psi_k(x) \otimes \psi_k(y))h_j \rangle_{\sH}\, dxdy
\nonumber \\
&\,\,=\,\,
\sum\limits_{k,\ell=m+1}^{n}\,\, 
\mu_\ell  \mu_k
\int_X \langle \psi_\ell(x),\psi_k(x)\rangle_\sH \,dx
\int_X \sum\limits_j \langle  \psi_k(y),h_j\rangle_\sH
\langle h_j,\psi_\ell(y)\rangle_\sH \,dy
\nonumber \\
&\,\,=\,\, 
\sum\limits_{\ell=m+1}^{n}\mu_\ell^2.
\end{align}
Therefore $P_n$ is Cauchy and the eigen-expansion converges in 
$L^2(X\times X;\cB_2(\sH))$ as $\cK$ is Hilbert-Schmidt.
Finally, $P$ is the kernel for $\cP$ by Theorem~\ref{thm:HSopB2ker}.
\end{proof}

\begin{remark}
Even if the kernel, $K$, is continuous, there is no guarantee that $P$
is continuous. Consequently, even in the case of scalar-valued kernels,
we cannot apply  Mercer's Theorem to the operator $\cP$. 
Indeed, a priori we do not have any control of the uniform (or even pointwise) convergence of the series \eqref{eq:TCPkernelExpansion}. Consequently,
even though each term in the series is continuous, we cannot conclude that 
$P$  is  continuous.
Nevertheless,  Weidmann~\cite{weidmann1966integraloperatoren} 
proved that if  a scalar-valued kernel, $K$, satisfies the H\"older continuity
condition \eqref{eq:holdconddef}, 
 then  there is a continuous kernel, $P_0$, so that $P(x,y) =P_0(x,y)$ for almost all $(x,y)\in X\times X$. 
This result enabled Weidmann 
 to  apply Mercer's Theorem to conclude that $\cP$ is trace class
and hence so is $\cK$. 
In the case that $\sH$ is infinite dimensional, to apply Mercer's theorem
we also need to impose the condition \eqref{eq:TraceClassInfDimCptCond}.
The following example motivates the need for this condition.
We have not been able to replace \eqref{eq:TraceClassInfDimCptCond}
 by a condition directly on $K(x,y)$.  
\end{remark}

\begin{example}
Building on Example~\ref{ex:MercerCounterExample}, let $\cK \in 
\cB_2(L^2([0,1];\ell_2))$ be given by
\begin{equation}
\cK \,\,=\,\,\sum\limits_{n=1}^\infty \frac 1n e_n \otimes e_{n+1}.
\end{equation}
Since the kernel for $\cK$ is a constant, independent of $(x,y)$, 
the condition~\eqref{eq:holdconddef} holds. 
However, since
 $\cP = (\cK\cK^*)^{1/2}$ is given by \eqref{eq:MercerCounterExample}, 
 the singular values
of $\cK$ are $\mu_n = \frac 1n$. Consequently, $\cK \notin \cB_1(L^2([0,1];\ell_2))$.
\end{example}

\begin{proposition}\label{prop:intKintP}
With notation as above
\begin{equation}
\int_X \| K(x,y) - K(x',y)\|^2_{\cB_2(\sH)}\, dy 
\,\,=\,\,
\int_X \| P(x,y) - P(x',y)\|^2_{\cB_2(\sH)}\, dy.
\end{equation}
\end{proposition}

\begin{proof}
We have that
\begin{align}
&\int_X \| K(x,y) - K(x',y)\|^2_{\cB_2(\sH)}\, dy 
 \nonumber \\
\,\,&=\,\,
\int_X \left\| \sum\limits_{\ell=1}^\infty \mu_\ell [ \psi_\ell(x) - \psi_\ell(x')]
\otimes \phi_\ell(y) \right\|^2_{\cB_2(\sH)}dy
\nonumber \\
\,\,&=\,\,
\int_X \sum\limits_j \left\| \sum\limits_{\ell=1}^\infty 
 \mu_\ell [ \psi_\ell(x) - \psi_\ell(x')] \langle \phi_\ell(y),h_j\rangle_\sH \right \|_\sH^2
 dy
 \nonumber \\
\,\,&=\,\,
\sum\limits_{k,\ell=1}^\infty \mu_\ell \mu_k 
\langle \psi_\ell(x) - \psi_\ell(x')\, , \, \psi_k(x) - \psi_k(x') \rangle_\sH
\int_X \sum\limits_j  \langle h_j, \phi_\ell(y) \rangle_\sH 
  \langle \phi_k(y) , h_j \rangle_\sH\, dy
  \nonumber \\
\,\,&=\,\,
\sum\limits_{\ell=1}^\infty \mu_\ell^2 \| \psi_\ell(x) - \psi_\ell(x') \|_\sH^2,
\end{align}
with the same result holding for $P$. 
\end{proof}

\begin{theorem}\label{thm:Weidmann}
 Suppose that $p>1$ and 
$\max(1, p/2) < q \leq p$. Let  
$X \subset \mathbb R^m$ be compact and
$P\in L^2(X\times X;\cB_2(\sH))$
be a kernel such  that
\begin{align}
\int_X \| P(x,y) \,-\, P(x',y)\|^p_{\cB_2(\sH)}\, dy \,\,&\leq\,\, C \|x-x'\|_{\mathbb R^m}^q
\label{eq:Weidmann1}
\\
\int_X \| P(x,y) \,-\, P(x,y')\|^p_{\cB_2(\sH)}\, dx \,\,&\leq\,\, C \|y-y'\|_{\mathbb R^m}^q,
\label{eq:Weidmann2}
\end{align}
for all $x,x',y,y'\in X$. Then there is a kernel
$P_0 \in C^{0}(X\times X;\cB_2(\sH))$
so that $P_0(x,y)=P(x,y)$ for almost all $(x,y)\in X\times X$. 
Furthermore, if $P\in L^\infty(X\times X; \cB_1(\sH))$, then
 $P_0 \in L^\infty(X\times X;  \cB_1(\sH))$.
\end{theorem} 

The proof of this result is based on ideas of Weidmann~\cite{weidmann1966integraloperatoren} 
and Brislawn~\cite{brislawn1988kernels}.

\begin{proof}
Let $B_r\subset\mathbb R^m$ be the ball of radius $r>0$ centered at origin
and let $\nu_r$ denote the $m$-dimensional volume of $B_r$.
Since $X$ is compact, $P\in L^1(X\times X;\cB_2(\sH))$ and so the 
local averages,
\begin{equation}\label{eq:LocalAv1}
P_r(x,y) \,\,=\,\, \frac{1}{\nu_r^2} \iint_{B_r\times B_r}
P(x+\xi,y+\eta)\, d\eta d\xi,
\end{equation}
of $P$ are defined and continuous. 
Our goal is to extract a sequence, $P_{r_n}$, that converges to a continuous
kernel, $P_0$, as $n\to\infty$.

Exactly as in the proof for scalar-valued kernels 
given by Weidmann~\cite{weidmann1966integraloperatoren}, we find that
\begin{align}
 \| P_r(x,y) \,-\, P_r(x',y)\|_{\cB_2(\sH)} \,\,&\leq\,\, C \|x-x'\|_{\mathbb R^m}^\beta
 \label{eq:WeidmannProof1}
 \\
 \| P_r(x,y) \,-\, P_r(x,y')\|_{\cB_2(\sH)} \,\,&\leq\,\, C \|y-y'\|_{\mathbb R^m}^\beta,
 \label{eq:WeidmannProof2}
\end{align}
where $\beta= \frac{q-1}p>0$ and 
the constant, $C$, is independent of $r$, $x$, $y$,$x'$ and $y'$.
The main idea in Weidmann's argument is that when
$\|x-x'\|_{\mathbb R^m} \leq r$ we can obtain \eqref{eq:WeidmannProof1}
by  applying H\"older's inequality to  the inner integral in \eqref{eq:LocalAv1}
and invoking \eqref{eq:Weidmann1}. An iterative argument is then used
to treat the case that $\|x-x'\|_{\mathbb R^m} >  r$.

First, the estimates \eqref{eq:WeidmannProof1} and \eqref{eq:WeidmannProof2}
can  be combined to show that $\{P_r\}_{r>0}$ is an equicontinuous
family of functions on $X\times X$ in the $\cB_2(\sH)$-norm.

Next, we introduce  
the Hardy-Littlewood maximal function 
 \begin{equation}
 M_pP(x,y) \,\,:=\,\, \sup\limits_{r>0} \frac{1}{\nu_r^2} \iint_{B_r\times B_r}
 \| P(x+\xi,y+\eta)\|_{\cB_p(\sH)}\, d\eta d\xi,
 \end{equation}
 for $p=1,2$. 
 To obtain a uniform bound on the kernels, $P_r$, we note that $M_2P$
 is finite almost everywhere since $\|P(\cdot,\cdot)\|_{\cB_2(\sH)}$ is integrable~\cite{stein1970singular}. 
Fix $(x_*,y_*)$ so that $M_2P(x_*,y_*) < \infty$.
Then, by \eqref{eq:WeidmannProof1} and \eqref{eq:WeidmannProof2},
\begin{align}
&\| P_r (x,y) \|_{\cB_2(\sH)}  \nonumber \\
&\,\,\leq\,\, \| P_r (x,y) - P_r (x,y_*) \|_{\cB_2(\sH)}  +
\| P_r (x,y_*) - P_r (x_*,y_*) \|_{\cB_2(\sH)} +
\| P_r (x_*,y_*) \|_{\cB_2(\sH)}  \nonumber  \\ 
&\,\,\leq\,\, 2C \operatorname{diam}(X)^\beta + M_2P(x_*,y_*),
\end{align}
and so the family $\{P_r\}_{r>0}$ is uniformly bounded. 
Therefore by the Arzela-Ascoli
Theorem~\cite[Theorem III.3.1]{lang1996real} there is a sequence $r_n\to0$ so that $P_{r_n} \to P_0$ converges
uniformly on $X\times X$ in the $\cB_2(\sH)$-norm. Since the functions $P_r$ are 
continuous in the $\cB_2(\sH)$ norm, $P_0$ is too.\footnote{However  
we cannot conclude that $P_0\in C^0(X\times X; \cB_1(\sH))$.}
 In addition, by the Lebesgue Differentiation Theorem,
$P_{r_n} \to P$ almost everywhere. Consequently, $P=P_0$ almost everywhere.
This completes the proof in the case that $\dim \sH < \infty$.

When $\dim \sH = \infty$ it remains to show that
$P_0 \in L^\infty(X\times X;\cB_1(\sH))$.
Fix $(x_0,y_0)\in X\times X$. Since $P_r\to P_0$ uniformly
on $X\times X$ in the $\cB_\infty(\sH)$ (i.e., operator) norm and
since $P_0$ is uniformly continuous in the $\cB_\infty(\sH)$ norm, for any sequence
$(r_n,x_n,y_n) \to (0,x_0,y_0)$, we have that 
$P_{r_n}(x_n,y_n) \to P_0(x_0,y_0)$ in the $\cB_\infty(\sH)$ norm. 
Therefore, by \cite[Theorem I.4.2]{gohberg1969introduction}, 
\begin{equation}
\mu_k(r_n,x_n,y_n) \to \mu_k(0,x_0,y_0),
\end{equation}
 where $\mu_k(r,x,y)$ denotes the $k$-th singular value of the operator
$ P_r(x,y)$.

Next, suppose that $P\in L^\infty(X\times X; \cB_1(\sH))$. 
By the Hardy-Littlewood Maximal Theorem~\cite{stein1970singular} there is a $C>0$ so that
$\|M_1 P\|_{L^\infty(X\times X;\mathbb R)} 
\leq C \|P\|_{L^\infty(X\times X;\cB_1(\sH))}$.
Hence there is a measure zero set, $Z$, so that 
 \begin{equation}\label{eq:HLMFZ}
 \| P_r(x,y) \|_{\cB_1(\sH)}
\leq  C\|P\|_{L^\infty(X\times X;\mathbb R)} < \infty
\qquad \forall (x,y)\notin  Z \text{ and } \forall r>0.
\end{equation}
By choosing the sequence,  $(r_n,x_n,y_n)\to (0,x_0,y_0)$ so that 
 $(x_n,y_n)\notin Z$, and applying Fatou's Lemma   and \eqref{eq:HLMFZ},
 we conclude that
 \begin{align}
 \| P_0(x_0,y_0) \|_{\cB_1(\sH)} \,\,&=\,\,
 \sum\limits_{k=1}^\infty \mu_k(0,x_0,y_0) 
\\
\,\,&\leq\,\,  \liminf_{n\to\infty} \sum\limits_{k=1}^\infty \mu_k(r_n,x_n,y_n)
\leq C \| P \|_{L^\infty(X\times X;\cB_1(\sH))} < \infty,
 \end{align} 
 as required.
\end{proof}

\begin{proof}[Proof of Theorem~\ref{thm:TraceClassInfDimCpt}]
By \eqref{eq:holdconddef} and Proposition~\ref{prop:intKintP}, the assumptions
\eqref{eq:Weidmann1} and \eqref{eq:Weidmann2} of Theorem~\ref{thm:Weidmann} hold for the kernel $P$ of $\cP = (\cK\cK^*)^{1/2}$
with $p=2$ and $q=2\gamma$.
Therefore $\cP$ has a continuous representative, $P_0\in C^0(X\times X;\cB_2(\sH))$. Furthermore, if $\dim\sH = \infty$, by the assumption in \eqref{eq:TraceClassInfDimCptCond}, Theorem~\ref{thm:Weidmann}, 
and the compactness of $X$, 
 we have that $\restr{P_0}{\Delta}\in L^1(\Delta;\cB_1(\sH))$.
Therefore, by Mercer's Theorem~\ref{thm:MercerA}, 
$\cP$ is trace class and hence so is $\cK$, since they have the same singular values. 
\end{proof}

\section{The case of general matrix-valued kernels on the real line}
\label{sec:TraceClassProofonReals}

In this section, we consider the case of an exponentially decaying 
Lipschitz-continuous matrix-valued kernel on the real line, which is
of interest for applications to the stability of nonlinear waves.
By using a change of variables from a finite interval to the real line, 
we can transform  Theorem~\ref{thm:TraceClassInfDimCpt} to obtain an analogous result on 
$\mathbb R$. For simplicity, we assume that the
kernel is Lipschitz continuous rather than being H\"older continuous.
The reason for making this assumption is that 
Lipschitz continuous functions are  differentiable almost everywhere,
whereas functions that are merely H\"older continuous may not be.
Throughout this section $\|\cdot\|$ denotes a vector or matrix norm.

\begin{theorem}\label{TraceClassThmRealLine}
Let $K\in L^2(\mathbb R\times \mathbb R, \mathbb C^{n\times n})$ be a 
Lipschitz continuous,  matrix-valued kernel such that 
there is an $R>0$ so that for all $|x|$, $|y|>R$
\begin{equation}\label{KandDerivsexp}
     \operatorname{max} \{  \|{K}(x,y)\|, \|\partial_x {K}(x,y)\|, \|\partial_y {K}(x,y)\| \}
     \leq C e^{-\alpha|x-y|},
    \end{equation}
    for some $C, \alpha > 0.$ 
Let $\cK \in \mathcal{B}_2(L^2(\mathbb R,\mathbb{C}^n))$
be the Hilbert-Schmidt operator with kernel $K$. 
Then, $\mathcal{K} \in \mathcal{B}_1(L^2(\mathbb R, \mathbb{C}^n))$ is trace class.
\end{theorem}

\begin{proof}
We  will use a scaling 
 mapping, $\phi:(-1,1)\to\mathbb R$,
 to transform $\mathcal K$ to an operator, $\widetilde{\mathcal K}$,
on $(-1,1)$ that is given by 
$\widetilde{\mathcal K} = \mathcal U {\mathcal K} \mathcal U^{-1}$,
where $\mathcal U: L^2(\mathbb R,\mathbb C^n) \to L^2((-1,1),\mathbb C^n)$ is an isometry defined in terms of $\phi$.
We will show that if the kernel $K$ for ${\mathcal K}$ is Lipschitz-continuous, then so is the transformed kernel $\widetilde{K}$  for 
$\widetilde{\mathcal K}$. Therefore, by Theorem~\ref{thm:TraceClassInfDimCpt}, $\widetilde{\cK}$  is trace class. Consequently, 
${\mathcal K}$ is also trace class since, being related a similarity transform, 
the two operators have the same singular values.

Fix $\delta \in (0, \alpha/3)$, where $\alpha$ is the constant in \eqref{KandDerivsexp}.
We define the scaling transformation $\phi: (-1,1) \to \mathbb{R}$, 
by
\begin{equation}
    x = \phi(y) := \frac{1}{2 \delta} \log \frac{1+y}{1-y}.
\end{equation}
Then
\begin{equation}
    \phi'(y) = \frac{1}{\delta} \frac{1}{1-y^2} > 0,
\end{equation}
and the inverse map, $\phi^{-1} : \mathbb{R} \to (-1,1),$ is given by 
\begin{equation}
    y = \phi^{-1}(x) := \tanh(\delta x) = \frac{e^{\delta x} - e^{- \delta x}}{e^{\delta x} + e^{- \delta x}}.
\end{equation}
The  operator $\mathcal U:L^2(\mathbb{R},\mathbb C^n) \to L^2((-1,1),\mathbb C^n)$  is defined by
\begin{equation}
    (\mathcal U \mathbf f)(y) = (\phi'(y))^{1/2} \mathbf f(\phi(y)), \quad\mathbf f \in L^2(\mathbb{R},\mathbb C^n).
\end{equation}
and $\mathcal U^{-1}: L^2((-1,1),\mathbb C^n) \to L^2(\mathbb{R},\mathbb C^n)$
is given by
\begin{equation}
    (\mathcal U^{-1}\mathbf g)(x) = ((\phi^{-1})'(x))^{1/2} \mathbf g(\phi^{-1}(x)),
    \quad\mathbf g\in L^2((-1,1),\mathbb C^n).
\end{equation}
 These operators are isometries, since by the change of variables theorem,  \begin{equation}
    \|\mathcal U\mathbf f\|_{L^2((-1,1),\mathbb C^n)}^2 = \int_{-1}^1 
    \| \mathbf f(\phi(y))\|^2 \phi'(y) dy = 
    \int_{-\infty}^{\infty} \|\mathbf f(x)\|^2 dx = 
    \| \mathbf f\|_{L^2(\mathbb{R}, \mathbb C^n)}^2,
\end{equation}
and similarly for $\mathcal U^{-1}.$

Next we claim that the kernel for the operator
$\widetilde{\mathcal{K}} = U \mathcal{K} U^{-1}$ on 
$L^2((-1,1),\mathbb C^n)$ 
 is given by 
\begin{equation}\label{transkernel}
    \widetilde{{K}}(y,y') = (\phi'(y))^{1/2} {K}(\phi(y),\phi(y'))(\phi'(y'))^{1/2}, \, y, y' \in (-1,1).
\end{equation}
We verify \eqref{transkernel} by applying 
the chain rule and the change of variables $x' = \phi(y')$, to obtain
\begin{eqnarray}
    (\widetilde{\mathcal{K}}\mathbf g)(y) &=& \nonumber  (\phi'(y))^{1/2} \int_{-\infty}^{\infty} {K}(\phi(y), x')((\phi^{-1})'(x'))^{1/2} \mathbf g(\phi^{-1}(x'))dx'\\
    &=& (\phi'(y))^{1/2} \int_{-1}^1 {K}(\phi(y), \phi(y')) (\phi'(y'))^{1/2}\mathbf g(y') dy'.
\end{eqnarray}
Although $\phi'(y)$ has singularities at $y = \pm 1$,  $\widetilde{{K}}$ is continuous on $[-1, 1] \times [-1, 1],$ since
\begin{equation}\label{transKcont}
    \lim_{y \to \pm 1} \widetilde{{K}}(y,y') = 0 = \lim_{y' \to \pm 1} \widetilde{K}(y,y').
\end{equation}
To prove (\ref{transKcont}), we first observe that since $\alpha > \delta$,
\begin{equation}
    e^{-\alpha|\phi(y)|} = \exp \left( \frac{- \alpha}{2 \delta} \left|\log \frac{1+y}{1-y} \right| \right) = \begin{cases}
        \left(\frac{1-y}{1+y}\right)^{\alpha / (2 \delta)}, \,\, y \in [0,1),\\
        \left(\frac{1+y}{1-y}\right)^{\alpha / (2 \delta)}, \,\, y \in (-1,0].
    \end{cases}
\end{equation}
Let 
\begin{equation}
    F_{\pm}(y) := \frac{(1 \mp y)^{\alpha / 2 \delta}}{(1 \pm y)^{\alpha/ 2 \delta}(1 - y^2)^{1/2}}.
\end{equation}
Then, since $\exp(-\alpha |\phi(y) - \phi(y')|) \leq \exp(-\alpha |\phi(y)|) 
\exp(\alpha |\phi(y')|)$, 
\begin{equation}
    \lim_{y  \to  1} \|(\phi'(y))^{1/2} K(\phi(y),\phi(y'))(\phi'(y'))^{1/2}\| 
    \,\,\leq\,\, \frac{C}{\delta^2}\lim_{y \to 1} F_+(y) F_{-}(y'),
\end{equation}
and since $\delta < \alpha/2$, 
\begin{equation}
    \lim_{y \to 1} F_{+}(y) =  \lim_{y \to 1} \frac{(1-y)^{\alpha/(2 \delta)}}{(1-y^2)^{1/2}(1+y)^{\alpha / (2 \delta)}} = 0,
\end{equation}
and
\begin{equation}
    \lim_{y \to -1} F_{-}(y) =  \lim_{y \to -1} \frac{(1+y)^{\alpha/(2 \delta)}}{(1-y^2)^{1/2}(1-y)^{\alpha / (2 \delta)}} = 0.
\end{equation}
Therefore, for $y = \pm 1$ and $y' \in [-1, 1],$ 
\begin{equation}\label{translimK}
     \lim_{y  \to  \pm 1} \|(\phi'(y))^{1/2} {K}(\phi(y),\phi(y'))(\phi'(y'))^{1/2}\| = 0,
\end{equation}
and the same will hold for $y' = \pm 1,$ when $y \in [-1, 1].$
A similar argument holds for $\nabla \widetilde{{K}},$
since
\begin{eqnarray}
\nonumber    \partial_y \widetilde{{K}}(y,y') \,\,&=\,\, (\phi'(y))^{1/2} \partial_y {K}(\phi(y),\phi(y'))(\phi'(y'))^{1/2}
\\ \label{transdyK} 
& \,\,+\,\, [(\phi'(y))^{1/2}]' {K}(\phi(y),\phi(y'))(\phi'(y'))^{1/2}.
\end{eqnarray}
Since, by assumption, both $\partial_x {K}$ and $\partial_y {K}$ have the same exponential decay as ${K},$ the first term in \eqref{transdyK} converges to $0$ as $y \rightarrow \pm 1$ or $y' \rightarrow \pm 1$ just as in \eqref{translimK} above. For the second term, we  observe that since 
$\delta <  \alpha/3$,
\begin{eqnarray}
    \lim\limits_{y \to 1} \|[(\phi'(y))^{1/2}]' {K}(\phi(y),\phi(y'))\| &\leq \frac{C}{\sqrt{\delta}}\lim\limits_{y \to 1} \frac{y}{(1-y^2)^{3/2}} \frac{(1-y)^{\alpha / (2 \delta)}}{(1+y)^{\alpha / (2 \delta)})}
    \nonumber
    \\
    &= \frac{C}{\sqrt{\delta}} \lim\limits_{y \to 1} \frac{y(1-y)^{\frac{1}{2}(\frac{\alpha}{\delta}-3)}}{(1+y)^{\frac{1}{2}(\frac{\alpha}{\delta}+3)}} = 0.
\end{eqnarray}
Therefore, $\widetilde{K}$ is differentiable almost everywhere with bounded derivative, and hence is Lipschitz, as required.
\end{proof}

\section{Appendix}\label{sec:Appendix}

Here we outline an alternate proof of Theorem~\ref{thm:TraceClassInfDimCpt}, which is valid
in the case that $\dim \sH < \infty$.

\begin{proof}[Second Proof of  Theorem~\ref{thm:TraceClassInfDimCpt}]
The main idea for this proof is contained in Fredholm's 1903 paper~\cite{Fred1903}. Of course at that point,  Fredholm did not quite have the concept of a trace class operator. Rather he proved that if a scalar-valued 
kernel, $K: [a,b]\times[a,b]\to\mathbb C$, for an integral operator, $\mathcal K$, 
is H\"older continuous with H\"older exponent, $\gamma\in (\tfrac 12,1]$, and if 
\begin{equation}
b_n(\mathcal{K}) \,\,:=\,\, \frac{1}{n!} \int\limits_{[a,b]^n} \operatorname{det}[ K(x_\alpha, x_\beta)]_{\alpha,\beta=1}^n\, dx_1 \cdots dx_n,
\end{equation}
then the infinite series,
\begin{equation}
D_{\mathcal K}(z) \,\,:=\,\, \sum\limits_{n=0}^\infty b_n(\mathcal{K}) z^n,
\label{eq:FredholmsDefn}
\end{equation}
converges uniformly and absolutely to an entire function of the complex parameter, $z$. Fredholm then uses \eqref{eq:FredholmsDefn}
as the definition of his determinant. In modern parlance,
we note that if the operator is already known to be trace class, then the regular Fredholm determinant of $\mathcal K$
is given by
$\det(\mathcal I + z \mathcal K) \,\,=\,\, D_{\mathcal K}(z)$.
To prove that the series \eqref{eq:FredholmsDefn} converges
Fredholm used an ingenious combination of estimates to show that 
there is a constant, $C_1$ so that
\begin{equation}
|b_n(\mathcal{K})| \,\, \leq \,\, \frac{C_1^n}{n!} \,n^{- \gamma + 1/2}.
\label{eq:bnestimate}
\end{equation}
We note that this estimate only holds for operators defined on a 
finite interval, not on the entire real line.

Gohberg, Goldberg, and Krupnik~\cite{GGK} use this estimate to
prove that if the operator $\mathcal K$ is Hermitian symmetric and 
if the H\"older exponent satisfies $\gamma > 1/2$, then $\mathcal K$
is trace class. Their proof is based on two main ideas. The first idea
is to show that 
$D_{\mathcal K}(z)$ is the limit in an appropriate sense of a 
sequence of finite dimensional determinants, $\operatorname{det}(I + z K_m)$, for $m\in \mathbb N$.
Consequently, the set of eigenvalues, $\{\lambda_j\}$, of the
Hermitian symmetric operator $\mathcal K$ coincides with the
set $\{-1/z_j\}$, where $\{z_j\}$ is the set of  zeros of the entire function $D_{\mathcal K}(z)$. 
The second idea is to use a result from the theory of the distribution of the zeros of entire functions~\cite{Levin} to show that the series 
\begin{equation}
\sum\limits_j |\lambda_j| \,\,=\,\,\sum\limits_j \frac{1}{|z_j|}
\label{eq:TClambdaseries}
\end{equation}
converges if  the order of growth, 
\begin{equation}
\rho_D \,\,:=\,\, \overline{\lim\limits_{n\to\infty}} \,\frac{n \log n}{\log \frac{1}{|b_n|}},  
\end{equation}
of the entire function $D_{\mathcal K}(z)$
satisfies $\rho_D < 1$. By \eqref{eq:bnestimate}, this inequality 
holds provided that $\gamma > 1/2$.
Finally we observe that if $\mathcal K$ is Hermitian then $|\lambda_j| = \mu_j$,
and so  the convergence of~\eqref{eq:TClambdaseries} implies that $\mathcal K$ is trace class. 

To extend this proof to matrix-valued kernels without making the additional assumption that the operator is Hermitian symmetric, we first note that if 
Theorem~\ref{thm:TraceClassInfDimCpt} holds for Hermitian operators, then it holds for operators that are not assumed to have any symmetry properties. 
To see this, let
\begin{equation}
    H \,\,:=\,\, \tfrac12 (K + K^*)
    \qquad\text{and}\qquad
      S \,\,:=\,\, \tfrac12 (K - K^*)
\end{equation}
be the Hermitian and skew-Hermitian parts of $K$. 
Since the kernel $\widetilde{S} = i S$ is Hermitian,
we see that $K = H - i \widetilde{S}$ is a linear combination
of Hermitian kernels, each of which is H\"older continuous.
Since the space of trace class operators is a vector space, we conclude that if the result is true for Hermitian operators, then it is true in general.

One of the challenges in the approach of Gohberg, Goldberg, and Krupnik
is that they had to  develop a theory of determinants of compact operators that is parallel to but distinct from the theory of regular and 2-modified Fredholm determinants of trace class and Hilbert-Schmidt operators. They use their theory  to 
show that $D_{\mathcal K}(z)$ is the limit of a 
sequence of finite dimensional determinants and that $\lambda_j = -1/z_j$.

Rather than relying on this theory, since we already know that $\mathcal K$ is
Hilbert-Schmidt, we can replace $D_{\mathcal K}(z)$ in \eqref{eq:FredholmsDefn}
with the 2-modified Fredholm determinant, which in the case of a matrix-valued kernel, is the entire function 
\begin{equation}
{\det}_2 (\mathcal I + z \mathcal K) \,\,:=\,\,
\sum\limits_{n=0}^\infty \sum\limits_{\mathbf j \in J^{(n)}_k} b_{n,\mathbf j}(\mathcal K) \, z^n,
\end{equation}
where 
$ J^{(n)}_k = \{ (j_1,\dots,j_n\, : \, 1 \leq j_\alpha \leq k, \forall \alpha \}$
is a multi-index set of cardinality $| J^{(n)}_k|=k^n$, and  
\begin{equation}\label{eq:TCbnj}
b_{n,\mathbf j}(\mathcal K) \,\,=\,\,
\frac{1}{n!} \int\limits_{[a,b]^n}
\det\left[ K_{j_\alpha j_\beta}(x_{\alpha},x_{\beta})(1-\delta_{\alpha\beta})\right]_{\alpha,\beta=1}^n
, dx_1\cdots dx_n.
\end{equation}
Using \eqref{eq:TCbnj} and much careful bookkeeping, it is possible to derive a version of Fredholm's estimate \eqref{eq:bnestimate}. Finally, we use \eqref{eq:TCdet2evalue} to complete the proof. 
\end{proof}


\section*{Acknowledgments} 
Y.L.  thanks the Courant Institute of Mathematical Sciences at NYU and 
especially Prof. Lai-Sang Young for their hospitality.  
The authors thank F. Gesztesy for pointing out relevant literature.

\bibliographystyle{elsarticle-num}
\bibliography{TraceClass}

\end{document}